\documentclass{article}
\usepackage[lang = british]{ems-rmi} 
\usepackage{enumitem}

\newtheorem{theorem}{Theorem}[section]
\newtheorem{lemma}[theorem]{Lemma}
\newtheorem{remark}[theorem]{Remark}
\newtheorem{definition}[theorem]{Definition}

\newtheorem{example}{Example}
\newtheorem{corollary}[theorem]{Corollary}
\newtheorem{assumption}{Assumption}

\numberwithin{equation}{section}

\begin{document}

\title{SDES WITH RANDOM AND IRREGULAR COEFFICIENTS}
\titlemark{SDEs WITH RANDOM COEFFICIENTS}

\emsauthor{1}{Guohuan Zhao}{G.~Zhao}


\emsaffil{1}{Faculty of Mathematics, Bielefeld University, 33615 Bielefeld, Germany\email{zhaoguohuan@gmail.com}}



\begin{abstract}
We consider It\^o uniformly nondegenerate equations with random coefficients. 
When the coefficients satisfy some low regularity assumptions with respect to the spatial variables and Malliavin differentiability assumptions on the sample points, the unique solvability of singular SDEs is proved by solving backward stochastic Kolmogorov equations and utilizing a modified Zvonkin type transformation. 
\end{abstract}

\maketitle


\section{Introduction}
The main purpose of this work is  to study  the well-posedness of stochastic differential equations (SDEs) with  random and irregular coefficients.  More precisely, we are concerned with the following SDE in $\mathbb{R}^n$: 
\begin{equation}\label{Eq-SDE}
X_t(\omega)=X_0(\omega)+ \int_0^t \sigma_s(X_s,\omega)\mathrm{d}  W_s(\omega)+\int_0^t b_s(X_s,\omega)\mathrm{d} s. 
\end{equation}
Here $\{ W_t\}_{t\in[0,1]}$ is a $d$-dimensional Brownian motion defined on a complete filtered probability space $(\Omega, {\mathscr F}, {\mathscr F}_t, {\mathbf P})$, ${\mathscr F}$ and ${\mathscr F}_t$ are generalized by $\{W_s\}_{s\in [0,1]}$ and $\{W_s\}_{s\in [0,t]}$, respectively. ${\mathscr B}$ is the Borel algebra on $\mathbb{R}^n$ and ${\mathscr P}$ is the collection of all the progressively measurable sets on $[0,1]\times \Omega$. The coefficients $\sigma: \mathbb{R}^n\times [0,1]\times \Omega \to \mathbb{R}^n\otimes \mathbb{R}^d$ and $b: \mathbb{R}^n\times [0,1]\times \Omega \to \mathbb{R}^n$ are ${\mathscr B}\times {\mathscr P}$-measurable.

In the past half century, a great deal of mathematical effort in stochastic analysis has been devoted to the study of the existence, uniqueness and regularity properties of strong solutions to It\^o uniformly nondegenerate stochastic equations with {\em deterministic} and irregular drifts. When $\nabla\sigma\in L^{2d}_{loc} $ and $b$ is bounded, Veretennikov \cite{veretennikov1980strong2} proved the strong existence and uniqueness of solutions to SDE \eqref{Eq-SDE} by developing the original idea proposed by Zvonkin in \cite{zvonkin1974transformation}. In the case that $\sigma={\mathbb I}$ and $b\in L^q_tL^p_x$ with $\frac{n}{p}+\frac{2}{q}<1$, using Girsanov's transformation and $L^q_tL^p_x$-estimate for parabolic equations, Krylov-R\"ockner \cite{krylov2005strong} obtained the existence and uniqueness of strong solutions to \eqref{Eq-SDE}. After that, a lot of works appeared to investigate  properties of the strong solution to \eqref{Eq-SDE} with singular drifts. Among all, we mention that the H\"older continuity of the stochastic flow was proved by Fedrizzi and Flandoli in \cite{fedrizzi2011pathwise}, provided that the coefficients meet the same condition in \cite{krylov2005strong}. When $b$ is bounded, Menoukeu et al.  \cite{menoukeu2013variational} obtained the weak differentiability of the stochastic flow and the Malliavin differentiability of $X_t$ with respect to the sample $\omega$ by using Malliavin's calculus.  Zhang  \cite{zhang2011stochastic} extended 
Veretennikov's unique strong solvability result to the case that $\nabla\sigma, b \in L^q_tL^p_x$ with $\frac{n}{p}+\frac{2}{q}<1$.  Under similar conditions, the regularities of strong solutions with respect to the initial data and sample point were also shown in \cite{ zhang2016stochastic} and \cite{xia2020lqlp}. For more recent results, we refer the reader to \cite{krylov2020strong} and \cite{rockner2021critical}. We also note that martingale problems and stochastic Lagrangian flows corresponding to \eqref{Eq-SDE} were studied by many researchers, among which we quote \cite{stroock1969diffusion, zhang2017heat, beck2019stochastic, zhang2020stochastic, zhao2019stochastic}.  

The well-posedness and regularities of strong solutions to SDEs with singular coefficients is not only a fundamental theoretical problem, but also has a wide range of applications in many mathematical and physical problems. For instance, in the remarkable paper \cite{flandoli2010well}, Flandoli, Gubinelli and Priola  studied the following linear stochastic transport equation  (see also \cite{flandoli2011random}): 
\begin{equation}\label{Eq-STE}
\partial_t u+ b\cdot\nabla u + \nabla u\circ \frac{\mathrm{d} W_t}{\mathrm{d} t}=0, \quad u_0=\varphi,  
\end{equation}
where $b: \mathbb{R}^n\times [0,1]\to \mathbb{R}^n$ is {\em deterministic}. Using the stochastic flow of the corresponding SDE (or stochastic characteristics), they proved the existence and uniqueness for the above equation in $L^\infty$-setting, provided that the drift $b$ is $\alpha$-H\"older continuous uniform in $t$ and the divergence of $b$ satisfies some integrability conditions. However, as mentioned in \cite{flandoli2010well},  one of the major obstacles to extending the regularization by noise phenomenon to the case where $b$ is random is the fact that even when $b$ is H\"older continuous in $x$ the stochastic characteristics corresponding to \eqref{Eq-STE} may not uniquely exist. Below is a simple but typical example:  
\begin{example}
Let $d=n=1$. Assume $\sigma=1$ and 
$$
b_t(x)=\sqrt{|x- W_t|}\wedge 1, \quad X_0=0.  
$$
Denote $Y_t:=X_t- W_t$, then $Y_t$ satisfies the following random ODE: 
$$
\mathrm{d} Y_t(\omega)= b_t(Y_t(\omega)+ W_t(\omega), \omega) \mathrm{d} t= \left(\sqrt{|Y_t(\omega)|} \wedge1\right )  \mathrm{d} t, \quad Y_0=0. 
$$
One can verify that $y^{(1)}_t\equiv 0$ and $y^{(2)}_t= \frac{t^2}{4}$ are two solutions of the above ODE, which implies $X_t^{(1)}= W_t$ and $X^{(2)}_t=\frac{t^2}{4}+ W_t$ are two ${\mathscr F}_t$-adapted solutions to equation 
$$
X_t=\int_0^t b_s(X_s) \mathrm{d} s+ W_t, \quad t\in [0,1]. 
$$
\end{example}

The above example demonstrates that  the non degeneracy of the noise and the uniformly H\"older continuity of $b_t(\cdot,\omega)$ are insufficient to guarantee the well-posedness of \eqref{Eq-SDE}. To the best of our knowledge, there is little literature was written to address this issue so far. The main work before this paper is \cite{duboscq2016stochastic}, which is a foundational work but misses some important requirements, in particular because it asks a specific form of Malliavin derivative for the drift, and in certain situations, $W^{1,p}$ regularity for drift with respect to $x$, which makes the results not so competitive. This paper attempts to make some progress in this direction. Roughly speaking, our main result, Theorem \ref{Th-SDE-Main}, shows that if the noise is additive and nondegenerate, and $b$ is H\"older in $x$, the well-posedness of the It\^o equation \eqref{Eq-SDE} is guaranteed  when $Db_t(x)$, the Malliavin differentiable of $b_t(x)$, also satisfies a H\"older continuity  assumption with respect to $x$. 

\medskip
With a little abuse of notation, in this paper $L^p(\Omega, {\mathscr F}, {\mathbf P}; \mathbb{R}^m)$ is abbreviated as $L^p(\Omega)$, where $m$ is an integer that may take different values in different places. Our main result is 
\begin{theorem}\label{Th-SDE-Main}
Let $\alpha\in (0,1)$, $p>n/\alpha$, $\Lambda> 1$, $\varDelta: =\left\{(s,t)\in [0,1]^2: 0\leqslant s\leqslant  t\leqslant 1\right\}$ and let $D$ be the Malliavin derivative operator. Assume that $\sigma$ and $b$  are  ${\mathscr B}\times {\mathscr P}$ measurable, then equation \eqref{Eq-SDE} admits a unique  solution if $\sigma$ and  $b$  satisfy the following assumptions: 
\begin{enumerate}[(i)]
\item 
for almost surely $\omega\in \Omega$, $\sigma(\omega)$ and $b(\omega)$ are bounded, and for all $x,y\in{\mathbb R}^n, t\in [0,1]$, 
\begin{align*}
 | b_t(x, \omega) -b_t(y, \omega) |  \leqslant \Lambda|x-y|^\alpha, \quad   |\sigma_t(x,\omega)-\sigma_t(y,\omega)| \leqslant \Lambda |x-y|;  
\end{align*}
\item for almost surely $\omega\in \Omega$ and all $(x,t)\in {\mathbb R}^n\times [0,1]$, 
\begin{align*}
\Lambda^{-1} |\xi|^2\leqslant \tfrac12 \sigma^{ik}_t\sigma^{jk}_t(x,\omega)\xi_i\xi_j\leqslant \Lambda|\xi|^2,\quad \forall \xi\in {\mathbb R}^d; 
\end{align*} 
\item for each $(x,t)\in {\mathbb R}^n\times[0,1]$, $\sigma_t(x), b_t(x)$ are Malliavin differentiable and the random fields $ D_s\sigma_t(x)$ and  $D_sb_t(x)$ have continuous versions as maps from ${\mathbb R}^n\times \varDelta$ to $L^{2p}(\Omega)$ such that 
\begin{align}\label{Eq-coe-Db}
\sup_{(s,t)\in\varDelta }  \left(\|D_s\sigma_t\|_{ C^\alpha({\mathbb R}^n;L^{2p}(\Omega))}+\|D_sb_t\|_{ C^\alpha({\mathbb R}^n;L^{2p}(\Omega))}  \right ) \leqslant \Lambda. 
\end{align}
\end{enumerate}

\end{theorem} 

We give an example of $b$ meeting the conditions in Theorem \ref{Th-SDE-Main}. 

\begin{example}
Let $n=d=1$,  $\alpha\in (0,1)$, $p>1/\alpha$. Assume $\bar b:[0,1]\times {\mathbb R}^2\to {\mathbb R}$ is bounded function satisfying 
$$
 \left(\left| \bar b_t(x,y)-\bar b_t(x',y)\right|+\left| \partial_y \bar b_t(x,y)- \partial_y \bar b_t(x',y)\right|\right ) \leqslant C {|x-x'|^\alpha}, \quad \forall x,x',y\in {\mathbb R}^n,\ t\in [0,1] 
$$
and 
$$
b_t(x,\omega):= \bar b_t\left(x, \int_0^t h_r(\omega) \mathrm{d}  W_r(\omega)\right). 
$$
Here $h$ is an adapted process satisfying 
$$ \sup_{s\in [0,1]} {\mathbf E} \left(|h_s|^{2p} + \int_0^1 | D_s h_r |^{2p} \mathrm{d} r\right)<\infty. 
$$ 
Noting that  
$$
D_s b_t(x) =\partial_y\bar b_t \left(x, \int_0^t h_r \, \mathrm{d}  W_r\right) \left(\int_s^t D_s h_r \, \mathrm{d}  W_r+ h_s\right){\mathbf{1}}_{\varDelta}(s,t), 
$$ 
by Burkholder-Davis-Gundy's inequality, one sees that    
\begin{align*}
&\sup_{t\in[0,1];\omega\in \Omega}\|  b_t(\cdot,\omega) \|_{C^\alpha({\mathbb R})} + \sup_{(s,t)\in \varDelta}\|D_s b_t\|_{C^\alpha({\mathbb R}; L^{2p}(\Omega))}\\
\leqslant &   C\left[1+ \sup_{s\in [0,1]} {\mathbf E} \left( |h_s|^{2p} + \int_0^1 | D_s h_r |^{2p} \mathrm{d} r\right)^{1/2p}\right ]<\infty, 
\end{align*}
so $b$ satisfies the conditions in (i) and (iii) in Theorem \ref{Th-SDE-Main}. 
\end{example}

Our approach of studying the well-posedness of \eqref{Eq-SDE} is using a modified Zvonkin transformation. Such kind of trick was  first proposed in \cite{zvonkin1974transformation} for solving SDEs with deterministic and bounded coefficients. To explain our main idea, let us first give a brief introduction to Zvonkin's idea. Denote 
$$
a=\frac{1}{2}\sigma\sigma^*, \quad L_t u=a_t^{ij}\partial_{ij}u +b^i_t\partial_i u. 
$$
When $a$ and $b$ are deterministic, $a, b\in L^\infty_t C^\alpha_x$ and $a$ is uniformly elliptic, by Schauder's estimate, the following backward equation 
$$
\partial_t u + L_t u =- b, \quad u_T(x)=0
$$
admits a unique solution $u\in L^\infty_t C^{2+\alpha}_x$ with $\partial_t u\in L^\infty_tC^\alpha_x$. Moreover, if $T$ is sufficiently small, the map $x\mapsto \phi_t(x):=x+u_t(x)$ is a $C^2$-homeomorphism. Assuming that $X_t$ solves \eqref{Eq-SDE}, by It\^o's formula, $Y_t:= \phi_t(X_t)$ satisfies a new SDE with Lipschitz continuous coefficients. Thus, the strong uniqueness of the solution to the original equation is given by the one of the new equation. Coming to the case that $\sigma, b$ are progressive measurable and 
$$
\mathrm{ess\,sup}_{\omega\in \Omega}( \| \sigma(\omega)\|_{L^\infty_t C^\alpha_x}+\| b(\omega)\|_{L^\infty_t C^\alpha_x})<\infty.
$$ 
Thanks to the classic Schuader's estimate, pointwisely, one can solve the backward equation
\begin{equation}\label{Eq-pde}
\partial_t w+L_t w+f=0, \quad w_T(x)=0. 
\end{equation}
Moreover, $w$ satisfies 
$$
\mathrm{ess\,sup}_{\omega\in \Omega} \left( \|w(\omega)\|_{L^\infty_t C^{2+\alpha}_x} + \|\partial_t w(\omega)\|_{L^\infty_tC^\alpha_x}\right )\leqslant C\mathrm{ess\,sup}_{\omega\in \Omega}\|f(\omega)\|_{L^\infty_tC^\alpha_x}.
$$ 
However, in this case for each $x\in {\mathbb R}^d$, the process $w_\cdot(\cdot, x): (t,\omega)\mapsto w_t(x,\omega)$ is non-adapted, so one cannot apply the It\^o-Wentzell formula as in the deterministic case.  A very natural way to overcome this difficulty is to consider the function $u_t:={\mathbf E}(w_t|{\mathscr F}_t)$ instead of $w_t$. Formally, $u_t$ satisfies the following backward stochastic Kolmogorov equation  (see Lemma \ref{Le-BSPDE}): 
\begin{equation}\label{Eq-Bspde}
\mathrm{d} u_t+(L_t u_t +f_t)\mathrm{d} t=v_t\cdot\mathrm{d} W_t, \quad u_T(x)=0. 
\end{equation}
On this point, a more general class of semi-linear equations including \eqref{Eq-Bspde} were already studied by Du-Qiu-Tang \cite{du2012lp}  in $L^p$-spaces and also by Tang-Wei \cite{tang2016cauchy} in H\"older spaces. However, the main obstacle of applying their result for our purpose is that one can only expect that the vector field $v$ is in some $L^p$ (or $C^\alpha$) space, which is far from enough to apply the It\^o-Wentzell formula  (see Lemma \ref{Le-Ito}). On the other hand, Duboscq-R\'eveillac \cite{duboscq2016stochastic},  studied the stochastic regularization effects of diffusions with random drift coefficients on random functions. After adding some Malliavin differentiability conditions on $b$ and $f$,  they extended the boundedness of time average of a deterministic function $f$  depending on a diffusion process $X$ with deterministic drift coefficient $b$ to random mappings $f$ and $b$ by investigate the backward stochastic Kolmogorov equation \eqref{Eq-Bspde} ($a\equiv{\mathbb I}$) in some $L^p$-type space. Inspired by \cite{duboscq2016stochastic} and \cite{zvonkin1974transformation}, in this paper we prove a $C^{2+\alpha}$ type estimate (Theorem \ref{Th-BSPDE-Holder}) for $(u,v)$, provided that the coefficients satisfy some Malliavin differentiability conditions. To achieve this purpose, we first extend the classic Schauder estimate to random PDEs with Banach variables. Such kind of extension gives the $C^{2+\alpha}$ estimate for $u$, as well as $C^\alpha$ estimate for $v$ (see Lemma \ref{Le-BSPDE}). The main ingredient of this paper is Theorem \ref{Th-BSPDE-Holder}, where we give the $C^{2+\alpha}$ estimate for $v$, provided that the Malliavin derivatives of the coefficients satisfy \eqref{Eq-coe-Db}.  To us, such kind of result is new and intriguing. With such regularity estimate in hand, we then use a modified It\^o-Wentzell's formula and Zvonkin type transformation to prove the well-posedness of \eqref{Eq-SDE}. We believe our results have the potential to be applied to stochastic transport equations with random coefficients and some other nonlinear stochastic PDEs. 

\medskip 

This paper is organized as follows: 
In Section 2, we investigate a  random Banach-valued  non-adapted Kolmogorov equation and prove its well-posedness in some H\"older type spaces. In Section 3, we study the solvability of  backward stochastic Kolmogorov equation \eqref{Eq-Bspde} in some $C^{2+\alpha}$ space. Our main result was proved in Section 4. A It\^o-Wenzell type formula and some auxiliary lemmas used in our main proofs were presented in Appendix. 

\section{Schauder Estimates for  Random  Banach-valued PDEs}
In this section, we give a self-contain proof of Schauder type estimate for random  Banch-valued parabolic PDEs by using Littlewood-Paley decomposition. 

\medskip

Let $T\in (0,1]$, $D$ be a domain of ${\mathbb R}^n$, $D_T=  D\times [0,T]$ and ${\mathcal B}$ be a real Banach space. For any $\alpha\in (0,1)$ and  strongly continuous  function $g: D \to {\mathcal B}$,  we define 
$$
\|g\|_{0;D}:= \sup_{x\in D} |g(x)|_{{\mathcal B}},  \quad [g]_{\alpha; D}:= \sup_{x,y\in D} \frac{|g(x)-g(y)|_{{\mathcal B}}}{|x-y|^\alpha}. 
$$
For $k\in {\mathbb N}$, 
$$
\|g\|_{C^{k+\alpha}(D; {\mathcal B})}:= \sum_{i=0}^k \|\nabla^i g\|_{0;D}+ [\nabla^k g]_{\alpha;D}
$$
Here and below, all the derivatives of an ${\mathcal B}$-valued function are defined with respect to the spatial variable in the strong sense, namely, $\nabla g$ is the unique map from ${\mathbb R}^n$ to ${\mathcal L}({\mathbb R}^n; {\mathcal B})$ such that $\lim_{|h|\to 0} |g(x+h)-g(x)-\nabla g(x) \cdot h|_{{\mathcal B}}=0$.  For any $\beta\geqslant 0$, the space $C^{\beta, 0}_{x,t}(D_T; {\mathcal B})$ consists all continuous  function $f: D_T \to {\mathcal B}$ such that 
$$
\|f\|_{C^{\beta, 0}_{x,t}(D_T; {\mathcal B})}:=\sup_{t\in [0,T]}\|f(t)\|_{C^{\beta}(D; {\mathcal B})}<\infty. 
$$
Below we always denote $Q_T={\mathbb R}^n\times [0,T]$ and $Q=Q_1$. If there is no confusion on the time parameter $T$ and underlying Banach space ${\mathcal B}$ , we simply write $C^\beta$ and $C^\beta_{x, t}$ instead of $C^\beta({\mathbb R}^n; {\mathcal B})$ and $C^{\beta, 0}_{x, t}(Q_T; {\mathcal B})$, respectively.

\subsection{Littlewood-Paley decomposition}

Let ${\mathscr S}({\mathbb R}^n)$ be the Schwartz space of all rapidly decreasing complex valued functions on ${\mathbb R}^n$, and ${\mathscr S}'({\mathbb R}^n)$ be the dual space of ${\mathscr S}({\mathbb R}^n)$ (tempered distribution space). Given $f\in {\mathscr S}({\mathbb R}^n)$, the Fourier transform and inverse transform of $f$ are defined by
$$
{\mathscr F} (f)(\xi):=(2\pi)^{-n/2}\int_{{\mathbb R}^n}\mathrm{e}^{-\mathrm{i}\xi\cdot x} f(x)\mathrm{d} x, 
$$
$$
{\mathscr F}^{-1}f(\xi):=(2\pi)^{-n/2}\int_{{\mathbb R}^n}\mathrm{e}^{\mathrm{i}\xi\cdot x} f(x)\mathrm{d} x. 
$$
Let $\chi:{\mathbb R}^n\to[0,1]$ be a smooth radial function with 
$$
\chi(\xi)=1,\ |\xi|\leqslant 1; \ \chi(\xi)=0,\ |\xi|\geqslant 3/2.
$$
Define
$$
\varphi(\xi):=\chi(\xi)-\chi(2\xi), \quad \varphi_{-1}(\cdot): =\chi(2\cdot), \quad 
\varphi_j(\cdot):= \varphi(2^{-j}\cdot)\,(j=0,1,2,\cdots). 
$$
It is easy to see that $\varphi\geqslant 0$ and supp $\varphi\subset B_{3/2}\setminus B_{1/2}$ and formally 
\begin{align}\label{EE1}
\sum_{j=-1}^k\varphi_j(\xi)=\chi(2^{-k}\xi)\stackrel{k\uparrow\infty}{\longrightarrow} 1.
\end{align}
In particular, if $|j-j'|\geqslant 2$, then
$$
\mathrm{supp}\varphi(2^{-j}\cdot)\cap\mathrm{supp}\varphi(2^{-j'}\cdot)=\varnothing.
$$
Let $\widetilde\varphi$ be another  smooth radial function,  $\mathrm{supp}\widetilde \varphi\in  B_{\frac{7}{4}}\backslash B_{\frac{1}{4}}$ and $\widetilde \varphi (x)=1$ for all $x\in  B_{\frac{3}{2}}\backslash B_{\frac{1}{2}}$. Denote
$$
h_j:= {\mathscr F}^{-1}(\varphi_j),\quad  \widetilde h_j:= {\mathscr F}^{-1}(\widetilde\varphi_j). 
$$
For any $f\in L^1({\mathbb R}^n; {\mathcal B})+L^\infty({\mathbb R}^n; {\mathcal B})$, define 
$$
\Delta_j f:=\int_{{\mathbb R}^n} h_j(x-y)f(y) \mathrm{d} y, \quad \widetilde \Delta_j f:=\int_{{\mathbb R}^n} \widetilde h_j(x-y)f(y) \mathrm{d} y. 
$$

\subsection{A basic apriori estimate}
Assume $(\Omega, {\mathscr F}, {\mathbf P})$ is a complete probability space, ${\mathcal H}$ is a real Hilbert spaces and ${\mathcal B}= L^p(\Omega, {\mathscr F}, {\mathbf P}; {\mathcal H})$ for some $p\geqslant 2$. Let $a^{ij}, b^i, c$ be real-valued measurable functions on $Q \times \Omega$ and define 
$$
L_t : =a_t^{ij}\partial_{ij} + b^i_t\partial_i +c_t.
$$ 
Fix $T\in (0,1]$,  we first give the precise definition of solutions to the following ${\mathcal B}$-valued PDE 
\begin{equation}\label{Eq-RPDE}
\begin{cases}
&\partial_t w+L_t w+f=0 \ \textit{ in } \ Q_T^o \\
& w_T=0\ \textit{ on }\  {\mathbb R}^n. 
\end{cases}
\end{equation}

\begin{definition}
A fucntion $w: Q_T\to {\mathcal B}$ is called a solution of \eqref{Eq-RPDE} if
\begin{enumerate}
\item For each $t\in[0,T]$,  $w(t, \cdot)$ is a twice strongly differentiable function from ${\mathbb R}^n$ to ${\mathcal B}$; 
\item For each $x \in {\mathbb R}^n$, the process $w(\cdot,x)$ is absolutely continuous from $[0,T]$ to ${\mathcal B}$ satisfying 
$$
w_t(x)=\int_t^T \big(L_s  w_s+f_s\big)(x)\mathrm{d} s. 
$$
\end{enumerate}
\end{definition}
\medskip

In order to study the solvability of \eqref{Eq-RPDE}, we need the following  
\begin{assumption}\label{Aspt1}
The map 
$(x, t,\omega)\mapsto \big(a_t(x,\omega), b_t(x,\omega), c_t(x,\omega), f_t(x,\omega)\big)$ is ${\mathscr B}(Q)\times {\mathscr F}$ measurable and there are constants $\alpha\in (0,1)$ and $\Lambda>1$ 
such that  for almost surely $\omega\in \Omega$, 
\begin{align}\label{Eq-con1}
 \| a^{ij}(\omega) \|_{C^{\alpha,0}_{x,t}}+\|  b^i(\omega) \|_{C^{\alpha,0}_{x,t}}+\|c(\omega)\|_{C^{\alpha,0}_{x,t}} \leqslant \Lambda  \tag{\bf H$_1$}, 
\end{align}
and 
\begin{align}\label{Eq-con2}
\Lambda^{-1} |\xi|^2\leqslant (\omega)\xi_i\xi_j\leqslant \Lambda|\xi|^2.  \tag{\bf H$_2$}
\end{align} 
\end{assumption} 
Our main result in this section is 
\begin{theorem}\label{Th-RPDE}
 Under Assumption \ref{Aspt1}, for any $f\in C^\alpha_{x,t}$, equation \eqref{Eq-RPDE} admits a unique solution $w$ in $C^{2+\alpha}_{x,t}$. Moreover,  
\begin{equation}\label{RPDE-Schuader}
\|\partial_t w\|_{C^\alpha_{x,t}}+ \|w\|_{C^{2+\alpha}_{x,t}}+ T^{-1}\|w\|_{C^{0}_{x,t}}    \leqslant C \|f\|_{C^\alpha_{x,t}},
\end{equation}
where $C$ only depends on $n, p, \alpha, \Lambda$. 
\end{theorem}

\medskip 
Like the proof for the classic Schauder estimate, we first consider the case  $a_t(x,\omega)=a_t(\omega)$ and $b=c=0$. Define 
$$
A_{t,s}:=\int_t^s a(r) \mathrm{d} r,\quad p^a_{t,s}(x):= (\det 4\pi A_{t,s})^{-1/2} \exp(-\langle x, A_{t,s}^{-1} x\rangle )
$$
and 
$$
P^a_{t,s} f(x):=\int_{{\mathbb R}^n} p^a_{t,s}(x-y) f(y)\mathrm{d} y. 
$$

\begin{lemma}\label{Le-RPDE}
Let $T\in (0,1], \alpha\in (0,1)$. Assume $a$ is $x$-independent and satisfies \eqref{Eq-con2}. For any $f\in C^\alpha_{x,t}$, the function $w_t(x)= \int_t^T P^a_{t,s} f_{s}(x) \mathrm{d} s$ is the unique function in $C^{2+\alpha}_{x,t}$  satisfying  
\begin{equation}\label{PDE0}
 w_t =\int_t^T (a^{ij}_s\partial_{ij} w_s + f_s ) \mathrm{d} s . 
\end{equation}
Moreover,  there is a constant $C$ only depends on $n,\alpha, p, \Lambda$ such that  
\begin{equation}\label{Basic-Est1}
\|\partial_t w\|_{C^{\alpha}_{x,t}} +\|w\|_{C^{2+\alpha}_{x,t}}+ T^{-1}\|w\|_{C^{\alpha}_{x,t}} \leqslant C \|f\|_{C^{\alpha}_{x,t}}.  
\end{equation}
\end{lemma}
\begin{proof} 
We first prove that the map $w$ defined above  satisfies \eqref{Basic-Est1} by using Littlewood-Paley decompositions. Recall that ${\mathcal B}= L^p(\Omega, {\mathscr F}, {\mathbf P}; {\mathcal H})$. 
For any $g\in L^1({\mathbb R}^n; {\mathcal B})+L^\infty({\mathbb R}^n; {\mathcal B})$, by Minkowski's inequality, we have  
\begin{equation}\label{eq-Pg}
\begin{aligned}
\|(\Delta_j  P^a_{t,s}g)(x)\|_{{\mathcal B}}= &\left({\mathbf E} |(\Delta_j  P^a_{t,s}g) (x)|_{{\mathcal H}}^p\right)^{1/p}= \left({\mathbf E} | (P^{a}_{t,s} \widetilde \Delta_j\Delta_j g) (x) |_{{\mathcal H}}^p\right)^{1/p}\\
=&\left[\int_\Omega \left|\int_{{\mathbb R}^n}(p_{t,s}^{a(\omega)}*\widetilde h_j)(y)\cdot \Delta_j g(x-y, \omega) \mathrm{d} y\right|_{{\mathcal H}}^p {\mathbf P}(\mathrm{d} \omega) \right]^{1/p}\\
\leqslant & \int_{{\mathbb R}^n} \mathrm{d} y \left[\int_\Omega |p^{a(\omega)}_{t,s}*\widetilde h_j (y)|^p \cdot |\Delta_j g(x-y,\omega)|_{{\mathcal H}}^p\  {\mathbf P}(\mathrm{d} \omega)\right]^{1/p}\\
\leqslant & \|\Delta_j g\|_0 \int_{{\mathbb R}^n} \left[ \mathrm{ess\,sup}_{\omega\in \Omega} |p_{t,s}^{a(\omega)}* \widetilde h_j(y)| \right]  \mathrm{d} y. 
\end{aligned}
\end{equation}
By \eqref{Eq-con2}, 
\begin{align*}
&\int_{{\mathbb R}^n} \left[ \mathrm{ess\,sup}_{\omega\in \Omega} |p_{t,s}^{a(\omega)}* \widetilde h_j(x)| \right]    \mathrm{d} x\leqslant  \left \| \sup_{{\mathbb I}/\Lambda\leqslant a\leqslant \Lambda{\mathbb I}} | p^a_{t,s}*\widetilde h_j(x) | \right\|_{L^1_x}\\
=&\int_{{\mathbb R}^n} \mathrm{d} x \  \sup_{{\mathbb I}/\Lambda\leqslant a\leqslant \Lambda{\mathbb I}} \left|\int_{{\mathbb R}^n} p_{t,s}^a(x-y) 2^{jn}\widetilde h_0(2^jy) \mathrm{d} y\right|\\
=&\int_{{\mathbb R}^n} \mathrm{d} x \  \sup_{{\mathbb I}/\Lambda\leqslant a\leqslant \Lambda{\mathbb I}} \left|\int_{{\mathbb R}^n} 2^{jn} p_{t,s}^{2^{2j}a}(2^j x-z) \widetilde h_0(z) \mathrm{d} z\right|\\
=&\int_{{\mathbb R}^n} \mathrm{d} x \  \sup_{{\mathbb I}/\Lambda\leqslant a\leqslant \Lambda{\mathbb I}} \left|\int_{{\mathbb R}^n}  p_{t,s}^{2^{2j}a}(x-z) \widetilde h_0(z) \mathrm{d} z\right|. 
\end{align*}
Noting that 
$$
\|f\|_{L^1} \leqslant C_{n, N} \|(1+|x|^{2N}) f(x)\|_{L^\infty}, \quad \forall N>n/2
$$ and 
$$
{\mathscr F}^{-1}(p^a_{t,s})(\xi)= \exp(-\langle \xi, A_{t,s} \xi\rangle ), 
$$
we obtain 
\begin{align*}
&\int_{{\mathbb R}^n} \left[ \mathrm{ess\,sup}_{\omega\in \Omega} |p_{t,s}^{a(\omega)}* \widetilde h_j(x)| \right]    \mathrm{d} x\leqslant \int_{{\mathbb R}^n} \mathrm{d} x \  \sup_{{\mathbb I}/\Lambda\leqslant a\leqslant \Lambda{\mathbb I}} \left|\int_{{\mathbb R}^n}  p_{t,s}^{2^{2j}a}(x-z) \widetilde h_0(z) \mathrm{d} z\right|\\
\leqslant & C  \left\| (1+|x|^{2N})  \sup_{{\mathbb I}/\Lambda\leqslant a\leqslant \Lambda{\mathbb I}} |p_{t,s}^{2^{2j}a}* \widetilde h_0| (x)\right\|_{L^\infty_x}\\
= & C  \sup_{{\mathbb I}/\Lambda\leqslant a\leqslant \Lambda{\mathbb I}} \left\| (1+|x|^{2N})   |p_{t,s}^{2^{2j}a}* \widetilde h_0| (x)\right\|_{L^\infty_x}\\
\leqslant & C \sup_{{\mathbb I}/\Lambda\leqslant a\leqslant \Lambda{\mathbb I}} \left\|(1+\Delta^N) [{\mathscr F}^{-1}(p^{2^{2j}a}_{t,s})\cdot {\mathscr F}^{-1} (\widetilde h_0) ](\xi) \right\|_{L^1_\xi}\\
= & C  \sup_{{\mathbb I}/\Lambda\leqslant a\leqslant \Lambda{\mathbb I}} \int_{B_{\frac{7}{4}}\backslash B_{\frac{1}{4}}} \left| (1+\Delta^N)[ \exp(-2^{2j}\langle \xi, A_{t,s} \xi\rangle )\cdot \widetilde \varphi](\xi)\right|  \ \mathrm{d} \xi. 
\end{align*}
Since $\sup_{|\alpha|=k} \partial^\alpha(\mathrm{e}^{a|\xi|^2}) \leqslant C (1+|a|)^k(1+|\xi|)^k \mathrm{e}^{a|\xi|^2}$, we get 
\begin{equation}\label{eq-p*h}
\begin{split}
&\int_{{\mathbb R}^n} \left[ \mathrm{ess\,sup}_{\omega\in \Omega} |p_{t,s}^{a(\omega)}* \widetilde h_j(x)| \right]    \mathrm{d} x\\
\leqslant & C\int_{\frac{1}{4} \leqslant |\xi|\leqslant \frac{7}{4} } [1+(\Lambda 2^{2j}(s-t) )^{2N}] \exp[{-2^{2j}(s-t)|\xi|^2/\Lambda}]  \mathrm{d} \xi. 
\end{split}
\end{equation}
Denote $\Lambda_j :=\Lambda 2^{2j}(s-t) $ and $\lambda_j :=\frac{1}{16}\Lambda^{-1} 2^{2j} (s-t) $. Combining \eqref{eq-Pg} and \eqref{eq-p*h}, we get 
\begin{align*}
\|\Delta_j  P^a_{t,s}g\|_{0}=\sup_{x\in{\mathbb R}^n}\|(\Delta_j  P^a_{t,s}g)(x)\|_{{\mathcal B}} \leqslant C (1+\Lambda_j^{2N}) \mathrm{e}^{-\lambda_j} \big|B_{\frac{7}{4}}\backslash B_{\frac{1}{4}} \big|\|\Delta_j g\|_0. 
\end{align*}
By Lemma \ref{Le-Chart-Holder} and the elementary inequality: 
$$(1+\Lambda_j^{2N}) \mathrm{e}^{-\lambda_j } \leqslant C_k (1\wedge [2^{2j}\cdot(s-t)]^{-k})\ (\forall k\in {\mathbb N}),
$$ 
we get 
\begin{align*}
\|\Delta_j P^a_{t,s} g\|_0\leqslant &C 2^{-j\alpha} \|g\|_\alpha (1+\Lambda_j^{2N}) \mathrm{e}^{-\lambda_j}\\
\leqslant & C_k 2^{-j\alpha}(1\wedge [2^{2j}\cdot(s-t)]^{-k}) \|g\|_\alpha. 
\end{align*}
This yields 
\begin{align*}
\|\Delta_j w_t\|_0 =&\left\|\Delta_j \int_t^T P^a_{t,s} f_{s} \ \mathrm{d} s\right\|_0  \\
\leqslant& C 2^{-j\alpha} \|f\|_{C^\alpha_{x,t}} \int_0^{T-t} (1\wedge 2^{-2jk} r^{-k}) \mathrm{d} r. 
\end{align*}
If $t \geqslant T-2^{-2j}$, then  
$$
\|\Delta_j w_t\|_0\leqslant  C 2^{-j\alpha} \|f\|_{C^\alpha_{x,t}}\cdot (T-t) \leqslant C 2^{-j(2+\alpha)} \|f\|_{C^\alpha_{x,t}}; 
$$
if $t<T-2^{-2j}$, by choosing $k=2$, then  
\begin{align*}
\|\Delta_j w_t\|_0\leqslant&  C 2^{-j\alpha} \|f\|_{C^\alpha_{x,t}}\cdot \left( 2^{-2j}+ 2^{-4j}\int_{2^{-2j}}^{T-t} s^{-2} \mathrm{d} s \right)\\
\leqslant &C 2^{-j(2+\alpha)} \|f\|_{C^\alpha_{x,t}}. 
\end{align*}
Again using Lemma \ref{Le-Chart-Holder}, one sees that 
$$
\|w\|_{C^{2+\alpha}_{x,t}}\leqslant C 
\sup_{\substack{t\in [0,T]\\ j\geqslant -1}} \left(2^{-j(2+\alpha)}\|\Delta_j w_t\|_0\right) \leqslant C \|f\|_{C^{\alpha}_{x,t}}. 
$$
So we complete our proof for \eqref{Basic-Est1}. By basic calculations, one can verify that $w$ satisfies \eqref{PDE0}. It remains to show that $w$ defined above is the unique solution to \eqref{Eq-RPDE} in $C^{2+\alpha}_{x,t}$. Assume $\widetilde w\in C^{2+\alpha}_{x,t}$ is another function satisfy \eqref{PDE0}. Let $0\leqslant \varrho\in C_c^\infty({\mathbb R}^n)$ satisfying $\int \varrho =1$ and $\varrho_\varepsilon (x)=\varepsilon^{-n} \varrho(x/\varepsilon)$. Define $v:= w-\widetilde w$ and $v^\varepsilon:= v* \varrho_\varepsilon$. For any $k>n/p$, $N>1$ and $\varepsilon\in (0,1)$,  by Sobolev embedding and H\"older's inequality, 
\begin{align*}
&{\mathbf E} \|v^\varepsilon_{t_1}-v^\varepsilon_{t_2}\|_{L^\infty(B_N; {\mathcal H})}^p={\mathbf E} \sup_{\|h\|_{{\mathcal H}=1}}\| \langle v^\varepsilon_{t_1}-v^\varepsilon_{t_2}, h\rangle \|_{L^\infty(B_N)}^p\\
 \leqslant &C
 N^{kp-n} {\mathbf E} \sup_{\|h\|_{{\mathcal H}}=1}\| \langle v^\varepsilon_{t_1}-v^\varepsilon_{t_2}, h\rangle \|_{W^{k, p}(B_N)}^p\\
 \leqslant & C
 N^{kp-n} {\mathbf E}  \sum_{i=0}^{k} \int_{B_N} \left|  \nabla^{i} \int_{t_1}^{t_2}   (a^{ij} \partial_{ij}v^\varepsilon_s) (x) \mathrm{d} s\right|_{{\mathcal H}}^p \mathrm{d} x\\
 \leqslant & C
 N^{kp-n} |t_2-t_1|^{p-1}  \sum_{i=2}^{k+2} \int_{B_N}\int_{t_1}^{t_2} {\mathbf E} \left|\int_{B_{N+1}} v_s(y) \nabla^i \rho_\varepsilon(x-y)\mathrm{d} y\right|_{{\mathcal H}}^p \mathrm{d} s \, \mathrm{d} x\\
 \leqslant & C_\varepsilon   
N^{kp+np-n} |t_2-t_1|^{p-1}  \int_{t_1}^{t_2}\int_{B_{N+1}} {\mathbf E} |v_s(y)|_{{\mathcal H}}^p\mathrm{d} y \\
 \leqslant & C_\varepsilon N^{(k+n)p} |t_2-t_1|^{p} \|v\|_{C^{0}_{x,t}}^p. 
\end{align*}
Due to Kolmogorov's criterion,   for almost surely $\omega\in \Omega$ and all $\varepsilon\in (0,1)$,  $(x,t)\in Q_T$ , 
$$
\| \langle v^\varepsilon_t(x, \omega)\|_{{\mathcal H}} \leqslant C_\varepsilon(\omega) (1+|x|)^{k+n}, 
$$
which means $v^\varepsilon_t(\cdot,\omega)$ satisfies a certain growth condition at infinity. On the other hand, by definition, for almost surely $\omega\in \Omega$ and each $h\in {\mathcal H}$, the real valued function $\langle v_t^\varepsilon(\omega), h\rangle $ satisfies 
$$
\partial_t \langle v_t^\varepsilon(\omega), h\rangle  +a_t^{ij}(\omega)\partial_{ij} \langle v_t^\varepsilon(\omega), h\rangle =0, \quad \langle v^\varepsilon_T(\omega), h\rangle =0. 
$$
Thus, 
we have $\langle v_t^\varepsilon(\omega), h\rangle \equiv 0$ (see \cite[Chapter 7, p176]{john1978partial}) i.e. $w*\varrho_\varepsilon=\widetilde w*\varrho_\varepsilon$ a.s.. So $\|w_t(x)-\widetilde w_t(x)\|_{{\mathcal B}}\leqslant \lim_{\varepsilon\to 0} \|w_t(x)-(w*\varrho_\varepsilon)_t(x)\|_{{\mathcal B}}+  \lim_{\varepsilon\to0} \|\widetilde w_t(x)-(\widetilde w*\varrho_\varepsilon)_t(x)\|_{{\mathcal B}}=0.$ So we complete our proof. 
\end{proof}

\begin{proof}[Proof of Theorem \ref{Th-RPDE}] 
Thanks to Lemma \ref{Le-RPDE} and the method of continuity, we only need to prove the aprior estimate  \eqref{RPDE-Schuader}. Assume $w\in C^{2+\alpha}_{x,t}$ is a solution to \eqref{Eq-RPDE}. Let $\chi\in C_c^\infty({\mathbb R}^d)$ so that $\chi(x)=1$ if $|x|\leqslant 1$ and $\chi(x)=0$ if $|x|\geqslant 2$. Fix a number $\delta>0$, which will be determined  later. Define $\chi^z_\delta=\chi((x-z)/\delta)$, then 
$$
\partial_t (w\chi^z_\delta) +L_t^z (w\chi^z_\delta)+(f\chi^z_\delta)+[\chi^z_\delta L_t w- L_t^z(w\chi^z_\delta)]=0, 
$$
where $L_t^z w_t(x):= a^{ij}_t(z) \partial_{ij} w_t(x)$. Using \eqref{Eq-con1} and noting that 
\begin{align*}
\chi^z_\delta L_t w- L_t^z(w\chi^z_\delta)=\chi^z_\delta (a^{ij}-a^{ij}_z)\partial_{ij} w+(b^i\chi^z_\delta-2 a^{ij}_z \partial_j \chi^z_\delta) \partial_i w+(c\chi^z_\delta-a^{ij}_z \partial_{ij} \chi^z_\delta) w,
\end{align*}
we have 
\begin{align}\label{Commutator1}
\begin{aligned}
&\| [\chi^z_\delta L_t w- L_t^z(w\chi^z_\delta)]\|_{C^\alpha_{x,t}}\\
\leqslant & C \delta^\alpha \|\nabla^2 w\|_{C^{\alpha, 0}_{x,t}(B_{2\delta}(z)\times[0,T]; {\mathcal B}))} + C\Big(  \delta^{-\alpha} \|\nabla^2 w\|_{C^0_{x,t}} \\
&+ \delta^{-1-\alpha} \|\nabla w\|_{C^\alpha_{x,t}} + \delta^{-2-\alpha} \|w\|_{C^\alpha_{x,t}} \Big). 
\end{aligned}
\end{align}
Combining Lemma \ref{Le-RPDE} and  equation \eqref{Commutator1}, we obtain that for any $\delta>0$, 
\begin{align*}
&\sup_{z\in {\mathbb R}^n} \|w\|_{C^{2+\alpha,0}_{x,t}(B_{2\delta}(z)\times[0,T]; {\mathcal B})}\leqslant C_n \sup_{z\in {\mathbb R}^n} \|w\|_{C^{2+\alpha,0}_{x,t}(B_\delta(z)\times[0,T]; {\mathcal B})}\\
\leqslant & C \sup_{z\in {\mathbb R}^n} \|w\chi^z_\delta\|_{C^{2+\alpha}_{x,t}} \leqslant C\sup_{z\in {\mathbb R}^n}\|f\chi^z_\delta+[\chi^z_\delta L_t w- L_t^z(w\chi^z_\delta)]\|_{C^{\alpha}_{x,t}}\\
\leqslant & C \delta^\alpha  \sup_{z\in {\mathbb R}^n} \|w\|_{C^{2+\alpha,0}_{x,t}(B_{2\delta}(z)\times[0,T]; {\mathcal B})}+C \Big(\delta^{-\alpha} \|\nabla^2 w\|_{C^{0}_{x,t}} + \delta^{-1-\alpha} \|\nabla w\|_{C^\alpha_{x,t}} \\
&+ \delta^{-2-\alpha} \|w\|_{C^\alpha_{x,t}} + \delta^{-\alpha} \|f\|_{C^\alpha_{x,t}}\Big). 
\end{align*}
By choosing $\delta\in (0,1)$ sufficiently small such that $C\delta^\alpha\leqslant 1/2$, we obtain 
\begin{align*}
\sup_{z\in {\mathbb R}^n} \|w\|_{C^{2+\alpha,0}_{x,t}(B_{2\delta}(z)\times[0,T]; {\mathcal B})}\leqslant  C_\delta \Big(\|w\|_{C^2_{x,t}} + \|f\|_{C^\alpha_{x,t}}\Big). 
\end{align*}
Using interpolation, we get 
\begin{align*}
\|w\|_{C^{2+\alpha}_{x,t}} \leqslant& C_\delta \sup_{z\in {\mathbb R}^n} \|w\|_{C^{2+\alpha,0}_{x,t}(B_{2\delta}(z)\times[0,T];{\mathcal B})}\\
\leqslant & \varepsilon C_\delta \|w\|_{C^{2+\alpha}_{x,t}} +C_{\delta,\varepsilon} \Big(\|w\|_{C^0_{x,t}} + \|f\|_{C^\alpha_{x,t}}\Big), \quad \forall\varepsilon\in (0,1). 
\end{align*}
By choosing $\varepsilon$ small such that $\varepsilon C_\delta\leqslant 1/2$, we get 
\begin{equation}\label{Est-step1}
\|w\|_{C^{2+\alpha}_{x,t}} \leqslant C \Big(\|w\|_{C^0_{x,t}} + \|f\|_{C^\alpha_{x,t}}\Big). 
\end{equation}
It remains to show that $\|w\|_{C^0_{x,t}}$ can be controlled by $\|f\|_{C^\alpha_{x,t}}$. By Minkowski's inequality, for any $t\in [0,T]$, 
 \begin{equation}\label{eq-wB1}
\begin{split}
&\left( {\mathbf E} \int_{B_r(x)} |w_t(y)|_{{\mathcal H}}^p\mathrm{d} y\right)^{1/p}= \left( {\mathbf E} \int_{B_r(x)} \left|\int_{t}^T \partial_s w_s(y) \mathrm{d} s\right|_{{\mathcal H}}^p\mathrm{d} y\right)^{1/p}\\
=& \left( {\mathbf E} \int_{B_r(x)} \left|\int_t^T (L_s w_s+f_s)(y) \mathrm{d} s\right|_{{\mathcal H}}^p\mathrm{d} y\right)^{1/p}\\
\leqslant & C \int_t^T \left( {\mathbf E} \int_{B_r(x)} |L_sw_s +f_s|_{{\mathcal H}}^p(y) \mathrm{d} y \right)^{1/p} \mathrm{d} s\\
\leqslant & C T r^{n/p} (\|w\|_{C^2_{x,t}}+\|f\|_{C^0_{x,t}}).  
\end{split}
\end{equation}
One the other hand,  by H\"older's inequality, 
\begin{equation}\label{eq-wB2}
\begin{split}
|w_t(x)|_{{\mathcal B}}\leqslant & \fint_{B_r(x)} |w_t(x)-w_t(y)|_{\mathcal B}\, \mathrm{d} y+  \fint_{B_r(x)} |w_t(y)|_{{\mathcal B}} \, \mathrm{d} y\\
\leqslant & \|\nabla w\|_{C^0_{x,t}} \fint_{B_r(x)}|x-y| \mathrm{d} y+ \fint_{B_r(x)} \left( {\mathbf E} \int_{B_r(x)} |w_t(y)|_{\mathcal H}^p\right)^{1/p}\,  \mathrm{d} y\\
\leqslant & r\|\nabla w\|_{C^0_{x,t}}+r^{-n/p}\left({\mathbf E} \int_{B_r(x)} |w_t(y)|_{\mathcal H}^p\,  \mathrm{d} y\right)^{1/p}.
\end{split}
\end{equation}
Combining \eqref{eq-wB1} and \eqref{eq-wB2}, we obtain 
$$
\|w\|_{C^0_{x,t}} \leqslant r \|\nabla w\|_{C^0_{x,t}}+ C T (\|w\|_{C^2_{x,t}}+\|f\|_{C^0_{x,t}}). 
$$
Due to \eqref{Est-step1}, 
$$
\| w\|_{C^2_{x,t}}\leqslant C(\|f\|_{C^\alpha_{x,t}}+\|w\|_{C^0_{x,t}}). 
$$
Combining the above two inequalities and letting $r\to 0$, we get  
$$
\|w\|_{C^0_{x,t}}\leqslant CT (\| w\|_{C^0_{x,t}}+\|f\|_{C^\alpha_{x,t}}).
$$
By choosing $T$ sufficiently small such that $CT \leqslant 1/2$, we get 
\[\|w\|_{C^0_{x,t}}\leqslant CT \|f\|_{C^\alpha_{x,t}}. \]
This together with \eqref{Est-step1} implies that \eqref{RPDE-Schuader} holds for some small $T>0$. The same estimate for arbitrary $T\in (0,1]$ can be obtained by induction. 
\end{proof}

\begin{remark}\label{Rek-Classic}
If $f$ satisfies 
$$
\mathrm{ess\,sup}_{\omega\in \Omega} \|f(\omega)\|_{C^{\alpha,0}_{x,t}(Q_T; {\mathbb R})}<\infty, 
$$
then \eqref{Eq-RPDE} can be solved  pointwisely and by the classic Schauder estimate, it holds that 
\begin{align}
\begin{aligned}
&\mathrm{ess\,sup}_{\omega\in \Omega}\Big( \|\partial_t w(\omega)\|_{C^{\alpha,0}_{x,t}(Q_T; {\mathbb R})}+ \|w(\omega)\|_{C^{2+\alpha,0}_{x,t}(Q_T; {\mathbb R})}\\
&+ T^{-1}\|w(\omega)\|_{C^{\alpha,0}_{x,t}(Q_T; {\mathbb R})}\Big) \leqslant  C \mathrm{ess\,sup}_{\omega\in \Omega}\|f(\omega)\|_{C^{\alpha,0}_{x,t}(Q_T; {\mathbb R})}. 
\end{aligned}
\end{align}
\end{remark}

\section{Schauder estimate for Backward SPDE}
In this section, we prove the solvability of \eqref{Eq-Bspde} in $C^{2+\alpha}_{x,t}\times C^{2+\alpha}_{x,t}$ space. Recall that $ W_t$ is a $d$-dimensional Brownian motion on a complete  probability space $(\Omega, {\mathscr F}, {\mathbf P})$, ${\mathscr F}_t= \sigma\{ W_s: s\leqslant t\}\bigvee \mathcal{N}$ and ${\mathscr F}={\mathscr F}_1$.  
For any $t\in [0,1]$ and $X\in {\mathscr F}$, we denote ${\mathbf E}^t X := {\mathbf E}(X|{\mathscr F}_t)$. 
Throughout this section, we always assume $T\in (0,1]$, ${\mathcal H}$ is a real Hilbert space, ${\mathcal B}=L^p(\Omega; {\mathcal H})$ for some $p\geqslant 2$ and $H=L^2([0,1]; {\mathbb R}^d)$. With a little abuse of notation, $L^p(\Omega)=L^p(\Omega; {\mathbb R}^m)$ for some integer $m\geqslant1$ that can be changed in different places. 

\begin{lemma}\label{Le-BSPDE}
Let  ${\mathcal H}={\mathbb R}$. Assume that $a,b,c$ are ${\mathscr B}\times {\mathscr P}$ measurable and satisfy Assumption \ref{Aspt1}, then the following BSPDE 
$$
u_t(x)= \int_t^T (L_s  u_s+ f_s)(x)\mathrm{d} s-\int_t^T v_s(x)\cdot \mathrm{d}  W_s
$$
has an ${\mathscr F}_t$-adapted solution $(u,v)$ in $C^{2+\alpha}_{x,t}\times C^{\alpha}({\mathbb R}^n; L^p(\Omega; H)$ and $u_t={\mathbf E}^t w_t$, where $w$ is the solution to \eqref{Eq-RPDE}. Moreover,  
$$
\|u\|_{C^{2+\alpha}_{x,t}}+ T^{-1}\|u\|_{C^{0}_{x,t}} +   \|v\|_{C^\alpha({\mathbb R}^n; L^p(\Omega; H))}  \leqslant C \|f\|_{C^\alpha_{x,t}},  
$$
where $C$ only depends on $n, d, p, \alpha, \Lambda$. 
\end{lemma}

\begin{proof}
Let $w$ be the solution of \eqref{Eq-RPDE}. Define $u_t(x)= {\mathbf E}^t w_t(x)$. By Theorem \ref{Th-RPDE} and Lemma \ref{Le-con-norm}, 
$$
\|u\|_{C^{2+\alpha}_{x,t}}+ T^{-1}\|u\|_{C^{0}_{x,t}}  \leqslant C \|f\|_{C^\alpha_{x,t}}. 
$$
Since $a_t(x), b_t(x)\in {\mathscr F}_t$, by the definitions of $u$, we have 
\begin{align*}
u_t(x)=&{\mathbf E}^{t}\left\{\int_t^T [(L_s  w_s+f_s)(x)]\mathrm{d} s\right\}=\int_t^T{\mathbf E}^{s} [(L_s w_s+f_s)(x)]\mathrm{d} s\\
&+\left\{\int_t^T{\mathbf E}^{t} [(L_s w_s+f_s)(x)]\mathrm{d} s-\int_t^T{\mathbf E}^{s} [(L_s w_s+f_s)(x)]\mathrm{d} s\right\}\\
=& \int_t^T (L_s  u_s+f_s)(x)\mathrm{d} s+m_t(x)-m_T(x). 
\end{align*}
Here 
\begin{equation}\label{eq-mt}
\begin{split}
m_t(x)=\int_t^T {\mathbf E}^{t} [(L_sw_s&+f_s)(x)]\mathrm{d} s\\
&+\int_0^t{\mathbf E}^{s} [(L_s  w_s+f_s)(x)]\mathrm{d} s\in {\mathscr F}_t. 
\end{split}
\end{equation}
For any $t\in[0,T]$, noting that 
\begin{align*}
{\mathbf E}^{t} m_T(x)=& {\mathbf E}^t \int_0^T {\mathbf E}^{s} [(L_s  w_s+f_s)(x)]\mathrm{d} s \\
=&{\mathbf E}^t \int_0^t {\mathbf E}^{s} [(L_s  w_s+f_s)(x)]\mathrm{d} s+ {\mathbf E}^t \int_t^T {\mathbf E}^{s} [(L_s  w_s+f_s)(x)]\mathrm{d} s \\
=&\int_0^t {\mathbf E}^{s} [(L_s  w_s+f_s)(x)]\mathrm{d} s+ \int_t^T {\mathbf E}^{t} [(L_s  w_s+f_s)(x)]\mathrm{d} s= m_t(x), 
\end{align*}
$m_\cdot(x)$ is a ${\mathscr F}_t$-martingale. By Theorem \ref{Th-RPDE}, \eqref{eq-mt} and Lemma \ref{Le-con-norm}, one can see that $m\in C^\alpha_{x,t}$. Thanks to  martingale representation, there is an ${\mathscr F}_t$-adapted process $v_\cdot(x)$ such that 
$$
m_t(x)-m_0(x)=\int_0^t v _s(x)\cdot\mathrm{d} W_s. 
$$
Hence, we have 
$$
u_t(x)= \int_t^T (L_s u_s+ f_s)(x)\mathrm{d} s- \int_t^T v_s(x)\cdot \mathrm{d}  W_s, 
$$
i.e. 
\begin{equation}\label{Eq-ut}
u_t(x)=u_0(x) - \int_0^t (L  u_s+ f_s)(x)\mathrm{d} s+ \int_0^t v_s(x)\cdot \mathrm{d}  W_s. 
\end{equation}
By \eqref{eq-mt} and Burkholder-Davis-Gundy inequality, we obtain 
\begin{align*}
&{\mathbf E} \left[\left(\int_0^T |v_t(x)-v_t(y)|^2 \mathrm{d} t\right)^{\frac{p}{2}} \right] \\
=&{\mathbf E} \langle m(x)-m(y)\rangle _T^{\frac{p}{2}}\leqslant C {\mathbf E} |m_T(x)-m_T(y)|^p \\
=& C {\mathbf E} \left| \int_0^T{\mathbf E}^{s} [(L_s  w_s+f_s)(x)-(L_s  w_s+f_s)(y)]\ \mathrm{d} s\right|^p\\
\leqslant & C \int_0^T   {\mathbf E} \left|{\mathbf E}^s \left[(L_s  w_s+f_s)(x)-(L_s  w_s+f_s)(y)\right]\right|^p \ \mathrm{d} s\\
\leqslant & C \int_0^T  {\mathbf E} |(L_s  w_s+f_s)(x)-(L_s  w_s+f_s)(y)|^p \ \mathrm{d} s\\
\leqslant & C |x-y|^{\alpha p} \left(\|w\|^p_{C^{2+\alpha}_{x,t}}+\|f\|^p_{C^{\alpha}_{x,t}}\right) \leqslant C |x-y|^{\alpha p} \|f\|^p_{C^{\alpha}_{x,t}}, 
\end{align*}
which yields 
$$
\|v\|_{C^\alpha({\mathbb R}^n; L^p(\Omega; H))} \leqslant C \|f\|_{C^{\alpha}_{x,t}}.  
$$
So we complete our proof. 
\end{proof}

\medskip

As we mentioned in the introduction, the Zvonkin type transform is an effective way to prove the well-posedeness of SDEs with singular coefficients. However, the $C^\alpha$-regularity of $v$ in the spatial variable is not enough to apply this trick. So we need to get better regularity estimate for $v$ under some mild conditions.  To achieve this goal, we start with some definitions and lemmas. Let ${\mathcal S}_b$ be the set of random variables of the form 
$$
F=f(\langle h_1, W\rangle , \cdots, \langle h_m, W\rangle ), 
$$
where $ f\in C_b^\infty({\mathbb R}^m)$, $h_i\in H$ and $\langle h_i, W\rangle := \int_0^1 h_s\mathrm{d}  W_s$. We define the operator $D$ on ${\mathcal S}_b$ with values in the set of $H$-valued random variables, by
$$
DF=\sum_{i=1}^m \partial_i f(\langle h_1, W\rangle , \cdots, \langle h_m, W\rangle ) h_i.
$$
For any $p\in [1,\infty)$, 
${\mathbb D}^{1,p}$ is the closure of the set ${\mathcal S}_b$ with respect to the norm 
$\|F\|_{{\mathbb D}^{1,p}}:= \|F\|_{p}+\|DF\|_{L^p(\Omega; H)}$. 
\begin{lemma}\label{Le-Ety}
Suppose $\{y_t\}_{t\in [0,1]}$ is a process (may not be adapted) on $(\Omega, {\mathbf P}, {\mathscr F})$ and 
$$
y_t=y_0+\int_0^t \dot y_r \mathrm{d} r,  
$$ 
with $y_0\in {\mathbb D}^{1,2}$ and $\dot y\in L^2([0,1]; {\mathbb D}^{1,2})$. Then there exists a random field $\{y_{s,t}\}_{(s,t)\in [0,1]^2}$ such that for each $t\in [0,1]$, $y_{\cdot, t}= D_\cdot y_t$ in $L^2(\Omega;H)$; for each $s\in [0,1]$, the map $[0,1]\ni t \mapsto y_{s,t} \in L^2(\Omega;{\mathbb R}^d)$ is absolutely  continuous and  
\begin{equation}\label{Eq-Etyt}
{\mathbf E}^t y_t={\mathbf E} y_0+ \int_0^t {\mathbf E}^s \dot y_s \mathrm{d} s +\int_0^t {\mathbf E}^s y_{s,s} \mathrm{d} W_s. 
\end{equation}
\end{lemma}
\begin{proof}
By our condition that $y_0\in {\mathbb D}^{1,2}$ and $\dot y\in L^2([0,1]; {\mathbb D}^{1,2})$, we have $D y_0\in L^2([0,1]\times \Omega;{\mathbb R}^d)$ and the map $(s,t,\omega)\mapsto D_s\dot y_t(\omega)$ is an element in $L^2([0,1]^2\times \Omega; {\mathbb R}^d)$. By Fubini's theorem, there is a Lebesgue null set ${\mathscr N}\subseteq [0,1]$ such that for each $s\notin {\mathscr N}$, the map $t\mapsto D_s \dot y_t$ is an element in $L^2([0,1]; L^2(\Omega))$ and $D_sy_0\in L^2(\Omega)$.  For any $s\in [0,1]$, define 
\begin{equation*}
y_{s,t}=
\begin{cases}
D_s y_0+ \int_0^t D_s \dot y_r \mathrm{d} r \ &s\notin {\mathscr N}, t\in[0,1]\\
0\ & s\in {\mathscr N}, t\in[0,1]. 
\end{cases}
\end{equation*}
Obviously, for each $s\in [0,1]$, the map $[0,1]\ni t\mapsto y_{t,s}\in L^2(\Omega)$ is absolutely continuous. By our assumption 
$$
\int_{0}^{1} \|\dot y_r\|_{{\mathbb D}^{1,2}} \mathrm{d} r \leqslant \left(\int_{0}^{1} \|\dot y_r\|_{{\mathbb D}^{1,2}}^2 \mathrm{d} r\right )^{1/2} <\infty, 
$$
i.e. $\dot y: [0,1] \to {\mathbb D}^{1,2}$ is Bochner  integrable. Since $D$ is a continuous operator from ${\mathbb D}^{1,2}$ to $L^2(\Omega)$, we get   
$$
Dy_t=D y_0+ D\int_0^t \dot y_r \mathrm{d} r =D y_0+ \int_0^t D \dot y_r \mathrm{d} r. 
$$
Combining this with the definition of $y_{s,t}$, we get $y_{\cdot,t}= D_\cdot y_t$ in $L^2(\Omega;H)$ for all $t\in [0,1]$. Moreover, by our assumption, 
\begin{align*}
{\mathbf E} \int_0^1 |y_{s,s}|^2 \mathrm{d} s\leqslant& {\mathbf E} \int_0^1 |D_s y_0|^2 \mathrm{d} s+ {\mathbf E} \int_0^1 \left| \int_0^s D_s \dot y_r \, \mathrm{d} r\right|^2 \mathrm{d} s\\
\leqslant & \|D y_0\|_2 + \int_0^T\|D \dot y_r\|_{2}^2 \mathrm{d} r<\infty, 
\end{align*}
which means $y_{s,s}$ is an element of $ L^2([0,1]\times \Omega;{\mathbb R}^d)$. By Lemma \ref{Le-COF},  we have 
\begin{equation}\label{eq-Etyt}
\begin{aligned}
{\mathbf E}^t y_t=& {\mathbf E} y_t+ \int_0^t {\mathbf E}^sD_sy_t \cdot \mathrm{d}  W_s  = {\mathbf E} y_t+ \int_0^t {\mathbf E}^s y_{s,t} \cdot \mathrm{d}  W_s  \\
=& {\mathbf E} y_t  +  \int_0^t {\mathbf E}^sy_{s,s} \cdot \mathrm{d}  W_s +  \int_0^t {\mathbf E}^s (y_{s,t}-y_{s,s}) \cdot \mathrm{d}  W_s. 
\end{aligned}
\end{equation}
Note that for any $s\notin {\mathscr N}$, $t\in[0,1]$, 
\begin{align*}
y_{s,t}-y_{s,s}=  \int_s^t D_s \dot y_r \mathrm{d} r
\end{align*}
by stochastic Fubini theorem, 
\begin{align*}
 &\int_0^t {\mathbf E}^s (y_{s,t}-y_{s,s}) \cdot \mathrm{d}  W_s =  \int_0^t {\mathbf E}^s \left( \int_s^t D_s \dot y_r \mathrm{d} r\right) \cdot \mathrm{d}  W_s\\
 =& \int_0^t  \left(  \int_s^t {\mathbf E}^sD_s  \dot y_r \mathrm{d} r\right) \cdot \mathrm{d}  W_s = \int_0^t \mathrm{d} r \int_0^r {\mathbf E}^sD_s \dot y_r \cdot \mathrm{d}  W_s \\
\overset{\eqref{eq-EtF}}{=}& \int_0^t ({\mathbf E}^r \dot y_r-{\mathbf E} \dot y_r) \mathrm{d} r = \int_0^t {\mathbf E}^r \dot y_r \mathrm{d} r + {\mathbf E} y_0-{\mathbf E} y_t. 
\end{align*}
Plugging this into \eqref{eq-Etyt}, we obtain \eqref{Eq-Etyt}. 
\end{proof}

For any $F\in {\mathscr F}$ and $h\in H$,  denote 
\begin{equation}\label{Eq-D-eps}
\tau_{\varepsilon h}F(\omega):= F\left(\omega+ \varepsilon\int_0^\cdot h_s \mathrm{d} s \right ), \quad D^h_\varepsilon  F:= \frac{(\tau_{\varepsilon h}F-F )}{\varepsilon}. 
\end{equation}
The next lemma is taken from \cite{mastrolia2017malliavin}, which gives a characterization of the space ${\mathbb D}^{1,p}$ in terms of differentiability properties. 
\begin{lemma}\label{Le-criterion}
Let $p\in (1,\infty)$ and $F\in L^p(\Omega)$. The following properties are equivalent
\begin{enumerate}
\item $F\in {\mathbb D}^{1,p}$. 
\item There is ${\mathcal D} F\in L^p(\Omega; H)$ such that for any $h\in H$ and $q\in [1,p)$ 
$$\lim_{\varepsilon\to 0}{\mathbf E}|D^h_\varepsilon F - \langle {\mathcal D} F, h\rangle _{H}|^q =0. $$
\item There is ${\mathcal D} F\in L^p(\Omega; H)$ and some $q\in [1,p)$ such that for any $h\in H$ 
$$\lim_{\varepsilon\to 0}{\mathbf E}|D^h_\varepsilon F - \langle {\mathcal D} F, h\rangle _{H}|^q =0. $$
\end{enumerate}
Moreover, in that case, $D F={\mathcal D} F$. 
\end{lemma}

Denote $\varDelta_T=\{(s,t): 0\leqslant s\leqslant t\leqslant T\}$, $\varDelta=\varDelta_1$. We need the following
\begin{assumption}\label{Aspt2}
 For each $(x,t)\in Q$, $a_t(x), b_t(x), c_t(x)$ are Malliavin differentiable and each of the random fields $ D_sa_t(x),  D_sb_t(x), D_sc_t(x)$ has a continuous version as a map from ${\mathbb R}^n\times \varDelta$ to $L^{2p}(\Omega)$ such that 
 \begin{align}\label{Eq-con3}
\begin{aligned}
\sup_{(s,t)\in \varDelta} \Big(\|D_sa_t\|_{C^\alpha({\mathbb R}^n;L^{2p}(\Omega))}&+ \|D_sb_t\|_{ C^\alpha({\mathbb R}^n;L^{2p}(\Omega))}\\
& +\|D_sc_t\|_{C^\alpha({\mathbb R}^n;L^{2p}(\Omega))}\Big) \leqslant \Lambda'<\infty. 
\end{aligned}
\tag{\bf H$_3$}
\end{align} 
\end{assumption}

The next Theorem is the key to the main purpose of this paper. 
\begin{theorem}\label{Th-BSPDE-Holder}
Let $T\in (0,1]$, $q>2p\geqslant 4$ and $C^{
\beta}_{x,t}=C^{\beta, 0}_{x,t}(Q_T; L^{p}(\Omega))$. Under Assumption \ref{Aspt1} and \ref{Aspt2}, the following BSPDE 
\begin{equation}\label{Eq-BSPDE}
u_t(x)= \int_t^T (L_s u_s+ f_s)(x)\mathrm{d} s- \int_t^T v_s(x)\cdot \mathrm{d}  W_s
\end{equation}
has an ${\mathscr F}_t$-adapted solution $(u,v)\in C^{2+\alpha}_{x,t}\times C^{2+\alpha}_{x,t}$, provided that $f\in C^{\alpha,0}_{x,t}(Q_T; L^q(\Omega))$ and $Df\in C^{\alpha,0}_{x,t}(Q_T; L^{2p}(\Omega; H))$. Moreover, there is a constant $C$ only depends on $n, d, p, q, \alpha,\Lambda, \Lambda'$ such that 
\begin{align*}
\|u\|_{C^{2+\alpha}_{x,t}}+\|v\|_{C^{2+\alpha}_{x,t}}\leqslant C \left( \|f\|_{C^{\alpha,0}_{x,t}(Q_T; L^{q}(\Omega))}+\sup_{(s,t)\in \varDelta_T} \|D_sf_t \|_{C^\alpha({\mathbb R}^n; L^{2p}(\Omega;{\mathbb R}^d))}\right ).
\end{align*} 
\end{theorem}
\begin{proof}
We divide the proof into four steps. 

{\em Step 1.} Let 
$$
\Lambda_f:= \|f\|_{C^{\alpha,0}_{x,t}(Q_T; L^{q}(\Omega))}+\sup_{(s,t)\in \varDelta_T} \|D_sf_t \|_{C^\alpha({\mathbb R}^n; L^{2p}(\Omega;{\mathbb R}^d))} 
$$
and $w$ be the unique solution to equation \eqref{Eq-RPDE} in $C^{2+\alpha,0}_{x,t}(Q_T; L^q(\Omega))$. Below we show that for each $(x,t)$, $w_t(x)$ is Malliavin differentiable, and $Dw$ satisfies the following $L^p(\Omega; H)$-valued equation: 
\begin{equation}\label{Eq-Dw}
Dw_t=\int_t^T (L_r D w_r+ G_r) \mathrm{d} r, 
\end{equation}
where $G_r=Df_r+(\partial_{ij}w_r Da^{ij}_r+\partial_i w^i_r D b^i_r + w_r\cdot Dc_r)$. To do this, we consider the following $L^p(\Omega; H)$-valued PDE, 
\begin{equation}\label{eq-ddw}
{\mathcal D} w_t=\int_t^T L_r ({\mathcal D} w_r) \mathrm{d} r + \int_t^T G_r\,  \mathrm{d} r =0. 
\end{equation}
By Assumption \ref{Aspt1}, \ref{Aspt2}  and Theorem \ref{Th-RPDE}, we get 
$$\|w\|_{C^{2+\alpha,0}_{x,t}(Q_T; L^{q}(\Omega))}\leqslant C\|f\|_{C^{\alpha,0}_{x,t}(Q_T; L^{q}(\Omega))}\leqslant C \Lambda_f,$$
$$
\sum_{i,j}\|D a^{ij}\|_{C^{\alpha, 0}_{x,t}(Q_T; L^{2p}(\Omega; H))}+\sum_{i}\|Db^i\|_{C^{\alpha, 0}_{x,t}(Q_T; L^{2p}(\Omega; H))}+\|Dc\|_{C^{\alpha, 0}_{x,t}(Q_T; L^{2p}(\Omega; H))}<\infty.  
$$
Recalling that $q>2p\geq 4$, H\"older's inequality yields, 
$$
\|G\|_{C^{\alpha,0}_{x,t}(Q_T; L^{p}(\Omega; H))}\leqslant C\Lambda_f. 
$$
Due to Theorem \ref{Th-RPDE} (with ${\mathcal H}=H$ therein), there is a unique solution ${\mathcal D} w\in C^{2+\alpha, 0}_{x,t}(Q_T; L^p(\Omega; H))$ sloves \eqref{eq-ddw}. Thus, for any $h\in H$, ${\mathcal D}^hw_t: = \langle {\mathcal D} w_t, h\rangle $ satisfies 
\begin{equation}\label{Eq-dhw}
{\mathcal D}^hw_t-\int_t^T L_r ({\mathcal D}^hw_r) \mathrm{d} r=\int_t^T  \langle G_r, h\rangle  \,  \mathrm{d} r 
\end{equation}
and 
\begin{align}\label{eq-Dhw}
\|{\mathcal D}^h w\|_{C^{2+\alpha}_{x,t}} +\| {\mathcal D}^h \partial_t  w\|_{C^{\alpha}_{x,t}} \leqslant C |h|_H\Lambda_f . 
\end{align}
Next we show that $D^h_\varepsilon w_t(x)$ (see \eqref{Eq-D-eps} for the definition) convergence to $\mathcal{D}^hw_t(x)$ in $L^p(\Omega)$, and as a consequence, we have $\mathcal{D} w_t(x)= Dw_t(x)$.  
By the definition of $D^h_\varepsilon w$, one sees  
\begin{align}\label{eq-Dhe-w}
\begin{aligned}
&D^h_\varepsilon  w_t - \int_t^T \Big[ \tau_{\varepsilon h} a^{ij}_r \,  \partial_{ij} D^h_\varepsilon  w_r+\tau_{\varepsilon h} b^i_r \, \partial_i D^h_\varepsilon  w_r+\tau_{\varepsilon h}c_r\, D^h_\varepsilon  w_r \Big] \mathrm{d} r\\
=&\int_t^T \Big[ D^h_\varepsilon  f_r  + D^h_\varepsilon  a^{ij}_r\,  \partial_{ij} w_r+D^h_\varepsilon  b^{i}_r \partial_{i} w_r + D^h_\varepsilon  \, c_r w_r \Big] \mathrm{d} r.
\end{aligned} 
\end{align}
Noting that for any $F\in {\mathbb D}^{1, p}$ and  $h\in H$, 
\begin{equation}\label{eq-DhF}
D^h_\varepsilon  F = \frac{(\tau_{\varepsilon h}F-F )}{\varepsilon}=\varepsilon^{-1}\int_0^\varepsilon \tau_{\theta h}\, D^h F \, \mathrm{d} \theta, 
\end{equation}
we get that for any $q'\in [p, 2p)$, 
\begin{align*}
{\mathbf E} |D^h_\varepsilon  f_r(x)-D^h_\varepsilon  f_r(y)|^{q'}=&\left\|\varepsilon^{-1}\int_0^\varepsilon \tau_{\theta h}\,[ D^h f_r(x)-D^hf_r(y)]  \mathrm{d} \theta\right\|_{L^{q'}(\Omega)}^{q'}\\
\leqslant & \sup_{0\leqslant \theta\leqslant \varepsilon} \|\tau_{\theta h} (D^h f_r(x)-D^h f_r(y))\|_{L^{q'}(\Omega)}^{q'}.  
\end{align*}
Due to Girsanov theorem, 
$$
\frac{\mathrm{d} {\mathbf P}\circ \tau_{\theta h}^{-1}}{\mathrm{d} {\mathbf P}}=
{\mathcal E}(\theta h): = \exp\left(\theta\int_0^T h_r\mathrm{d} W_r- \frac{\theta^2}{2} \int_0^T |h_r|^2\mathrm{d} r\right ).  
$$
Hence, 
\begin{align*}
{\mathbf E} |D^h_\varepsilon  f_r(x)-D^h_\varepsilon  f_r(y)|^{q'}\leqslant & \sup_{0\leqslant \theta \leqslant \varepsilon} {\mathbf E} [|D^h f_r(x)-D^h f_r(y)|^{q'} {\mathcal E}(\theta h)]\\
\leqslant & \sup_{0\leqslant \theta \leqslant \varepsilon} {\mathbf E} [|D^h f_r(x)-D^h f_r(y)|^{2p}]^{\frac{q'}{2p}} \cdot {\mathbf E}[ {\mathcal E}^{\frac{q}{2p-q'}}(\theta h)]^{1-\frac{q'}{2p}} \\
\leqslant& C \|Df_r\|_{C^\alpha({\mathbb R}^n; L^{2p}(\Omega; H))}^{q'} |h|_H^{q'} |x-y|^{\alpha q'},   
\end{align*}
where we use the following fact in the last inequality: 
$$
{\mathbf E} {\mathcal E}^\kappa(\theta h) = {\mathbf E} {\mathcal E}(\kappa\theta h) \exp\left(\frac{\kappa^2-\kappa}{2}|h|_{H}^2\right )\leqslant C_\kappa. 
$$ 
Thus, 
$$
\sup_{\varepsilon\in (0,1)} \| D^h_\varepsilon  f\|_{C^{\alpha,0}_{x,t}(Q_T; L^{q'}(\Omega))} \leqslant C |h|_H \|Df\|_{_{C^{\alpha,0}_{x,t}(Q_T; L^{2p}(\Omega; H))} }. 
$$
Similarly, for any $q''\in (1, 2p)$, 
\begin{align*}
\sup_{\varepsilon\in (0,1) } \Big[ \| D^h_\varepsilon  a\|_{C^{\alpha,0}_{x,t}(Q_T; L^{q''}(\Omega))}&+\| D^h_\varepsilon  b\|_{C^{\alpha,0}_{x,t}(Q_T; L^{q''}(\Omega))} \\
&+ \| D^h_\varepsilon  \, c_r\|_{C^{\alpha,0}_{x,t}(Q_T; L^{q''}(\Omega))}\Big] \leqslant C. 
\end{align*}
Choosing $q'=p$ and $q''=\frac{pq}{q-p}\in(p, 2p)$, and noticing that $\|w\|_{C^{2+\alpha,0}_{x,t}(Q_T; L^{q}(\Omega))}\leqslant C\Lambda_f$, by H\"older's inequality, we get 
\begin{equation}\label{eq-unibdd-1}
\begin{split}
&\sup_{\varepsilon\in (0,1)} \left\| D^h_\varepsilon  f  + D^h_\varepsilon  a^{ij}\,  \partial_{ij} w+D^h_\varepsilon  b^{i} \partial_{i} w + D^h_\varepsilon  \, c w \right\|_{C^{\alpha,0}_{x,t}(Q_T; L^{p}(\Omega))}\leqslant C|h|_H\Lambda_f. 
\end{split}
\end{equation}
Since $\tau_{\varepsilon h} a, \tau_{\varepsilon h} b, \tau_{\varepsilon h} c$ satisfy \eqref{Eq-con1} and \eqref{Eq-con2},  by \eqref{eq-Dhe-w}, \eqref{eq-unibdd-1} and Theorem \ref{Th-RPDE}, we have 
\begin{align}\label{eq-dhw-uni}
\sup_{\varepsilon\in (0,1)} \left(\|D^h_\varepsilon  w\|_{C^{2+\alpha}_{x,t}} +\| D^h_\varepsilon \partial_t  w\|_{C^{\alpha}_{x,t}} \right )\leqslant C|h|_H\Lambda_f.  
\end{align}
Let $\delta^{h}_\varepsilon w:= D^h_\varepsilon  w- {\mathcal D}^h w$. Next we want to prove $\delta^h_\varepsilon w_t(x)\to 0$ in $L^p(\Omega)$, for each $(x,t)\in Q_T$.  By definition, 
\begin{align}\label{eq-dhe-w}
\begin{aligned}
\partial_t\delta^{h}_\varepsilon w&+ L_t \delta^{h}_\varepsilon w = -(D^h_\varepsilon f- D^h f) \\
&- \left[ (D^h_\varepsilon  a^{ij}- D^h a^{ij}) \partial_{ij} w+(D^h_\varepsilon  b^{i}- D^h b^{i}) \partial_{i} w+ (D^h_\varepsilon  c- D^h c)  w\right ]   \\
&- \varepsilon \left( D^h_\varepsilon  a^{ij}\, \partial_{ij} D^h_\varepsilon  w+D^h_\varepsilon  b^{i}\, \partial_{i} D^h_\varepsilon  w+ D^h_\varepsilon  c\,D^h_\varepsilon  w \right )  =: -\sum_{i=1}^3F^{\varepsilon,i}
\end{aligned}
\end{align}
i.e. $\delta^{h}_\varepsilon w$ is a $L^p(\Omega)$-valued solution to \eqref{Eq-RPDE} with $f$ replaced by $F_t^\varepsilon:=\sum_{i=1}^3F_t^{\varepsilon,i}$. Estimates \eqref{eq-Dhw} and \eqref{eq-dhw-uni} yield 
\begin{align}\label{eq-Dhw-uni}
\sup_{\varepsilon\in (0,1)} \left(\|\delta^{h}_\varepsilon w\|_{C^{2+\alpha}_{x,t}} +\| \partial_t  \delta^{h}_\varepsilon w\|_{C^{\alpha}_{x,t}} \right )\leqslant C|h|_H\Lambda_f. 
\end{align}
By  \eqref{eq-dhe-w}, for each $R>0$, we have 
\begin{align*}
\partial_t (\delta^h_\varepsilon w \chi_R) &+ L_t (\delta^h_\varepsilon w\chi_R)+  F^{\varepsilon}\chi_R\\
&-(2a^{ij}\partial_{i}\delta^h_\varepsilon w\partial_j\chi_R+\delta^h_\varepsilon w a_t^{ij}\partial_{ij}\chi_R+\delta^h_\varepsilon wb^i_t\partial_i  \chi_R)=0,  
\end{align*} 
where $\chi_R(x)=\chi(x/R)$. Due to our assumptions and \eqref{eq-Dhw-uni}, 
$$
\left\|(2a^{ij}\partial_{i}\delta^h_\varepsilon w\partial_j\chi_R+\delta^h_\varepsilon w a_t^{ij}\partial_{ij}\chi_R+\delta^h_\varepsilon wb^i_t\partial_i  \chi_R)\right\|_{C^{\alpha}_{x,t}} \leqslant C|h|_H\Lambda_f /R. 
$$
So by Theorem \ref{Th-RPDE}, for any $\alpha'\in (0,\alpha)$, 
\begin{equation}\label{eq-dw}
\|\delta^h_\varepsilon w \chi_R\|_{C^{2+\alpha'}_{x,t}}\leqslant C \|F^{\varepsilon}\chi_R\|_{C^{\alpha'}_{x,t}}+ C|h|_H\Lambda_f/R.  
\end{equation}
Thanks to Lemma \ref{Le-criterion}, for each $(x,t)\in Q_T$, $F^{\varepsilon, 1}_t(x) = D^h_\varepsilon  f_t(x)-D^hf_t(x) \overset{L^{2p}(\Omega)}{\longrightarrow} 0$. By \eqref{eq-DhF} and the continuity of $Df: Q_T\mapsto L^{2p}(\Omega; H)$, one can verify that the map $Q_T\ni (x,t)\mapsto D^h_\varepsilon f_t(x)\in L^p(\Omega)$ is equivalent continuous. So by Arzela-Ascoli theorem, for any sequence $\varepsilon_n\to 0(n\to\infty)$, there exists a subsequence $\varepsilon_{n_k}\to 0(k\to\infty)$ such that for all $R>0$, $F^{\varepsilon_{n_k}, 1}\chi_R\to 0$ in $C^{\alpha'}_{x,t}$ with some $\alpha'\in (0,\alpha)$. Similarly, we have $F^{\varepsilon_{n_k}, 2}\chi_R\to 0$ and $F^{\varepsilon_{n_k}, 3}\chi_R\to 0$ in $C^{\alpha'}_{x,t}$ as $k\to \infty$. Thus, $\limsup_{\varepsilon\to0} \|F^\varepsilon\chi_R\|_{C^{\alpha'}_{x,t}}=0$. So by \eqref{eq-dw}, for any $R_0>0$, 
$$
\limsup_{\varepsilon\to0} \|\delta_\varepsilon^h w\chi_{R_0}\|_{C^{2+\alpha'}_{x,t}}\leqslant \lim_{R\to\infty}\limsup_{\varepsilon\to0} \|\delta_\varepsilon^h w\chi_R\|_{C^{2+\alpha'}_{x,t}} \leqslant \lim_{R\to\infty}C/R= 0, 
$$
which of course implies $D^h_\varepsilon  w_t(x)-{\mathcal D}^h w_t(x)\to 0$ in $L^p(\Omega)$. Again by Lemma \ref{Le-criterion}, for each $(x,t)\in Q_T$, we have 
$w_t(x)\in {\mathbb D}^{1,p}$ and $Dw_t(x)={\mathcal D} w_t(x)\in C^{2+\alpha,0}_{x,t}(Q_T; L^p(\Omega; H))$. Estimate \eqref{Eq-Dw} follows by the definition of ${\mathcal D} w$.

\medskip

{\em Step 2.} For any $(s,t)\in \varDelta_T$, let $w^s_t(x)$ be the solution to the following equation 
\begin{equation}
w^s_t=\int_t^T (L_r w^s_r+g^s_r) \mathrm{d} r, 
\end{equation}
where $g^s_r := (D_sa^{ij}_r)\partial_{ij}w_r+(D_sb^i_r) \partial_iw_r+(D_s c_r) w_r+D_sf_r$. 
By H\"older's inequality, 
\begin{align*}
\|g^s\|_{C^{\alpha}_{x,t}} \leqslant & \|D_sf \|_{C^{\alpha,0}_{x,t}(Q_T; L^{p}(\Omega))}+  \|w\|_{C^{2+\alpha,0}_{x,t}(Q_T; L^{2p}(\Omega))} \Big (\sum_{ij}\|D_sa^{ij} \|_{C^{\alpha,0}_{x,t}(Q_T; L^{2p}(\Omega))}\\
&+ \sum_{i}\|D_sb^i \|_{C^{\alpha,0}_{x,t}(Q_T; L^{2p}(\Omega))} +\|D_sc \|_{C^{\alpha,0}_{x,t}(Q_T; L^{2p}(\Omega))}\Big) \\
\leqslant & C  \|f\|_{C^{\alpha,0}_{x,t}(Q_T; L^{q}(\Omega))} + C \sup_{(s,t)\in \varDelta_T} \|D_sf_t \|_{C^{\alpha}({\mathbb R}^n; L^{p}(\Omega))}\leqslant C \Lambda_f. 
\end{align*}
Theorem \ref{Th-RPDE} yields, 
\begin{align}\label{eq-ws-norm}
\begin{aligned}
\sup_{s\in[0,T]} \left( \|\partial_t w^s\|_{C^{\alpha}_{x,t}}+\|w^s\|_{C^{2+\alpha}_{x,t}} \right )\leqslant  C \|g^s\|_{C^{\alpha}_{x,t}}\leqslant C \Lambda_f.
\end{aligned}
\end{align}

\medskip

{\em Step 3.} In this step, we prove that $w^s_t(x)$ constructed in {\em Step 2} is a  version of $D_sw_t(x)$. Let 
$$
{\mathscr A}^\alpha=\left\{w: w\in C^{2+\alpha}_{x,t}, \partial_t w\in C^{\alpha}_{x,t}\right\}, \quad \|w\|_{{\mathscr A}^\alpha}:=  \|w\|_{C^{2+\alpha}_{x,t}}+ \|\partial_t w\|_{C^{\alpha}_{x,t}}. 
$$
By linearity and Theorem \ref{Th-RPDE}, the solution map of \eqref{Eq-RPDE}
$$
{\mathcal T}: C^\alpha_{x,t} \ni f\mapsto w \in {\mathscr A}^\alpha
$$ 
is Lipschitz continuous. Since $[0,T]\ni s\mapsto g^s\in C^{\alpha}_{x,t}$ is measurable,  $s\mapsto w^s$ is measurable from $[0,T]$ to ${\mathscr A}^\alpha$. For any $\varphi \in C^\infty_c((0,T); {\mathbb R}^d)$, define 
$$
w^\varphi = \int_0^T \varphi(s) \cdot w^s \mathrm{d} s, \quad g^\varphi = \int_0^T \varphi(s) \cdot g^s \mathrm{d} s. 
$$
Then, one sees that $w^\varphi$ satisfies 
$$
w^\varphi_t=\int_t^T (L_r w^\varphi_r+g^\varphi_r) \mathrm{d} r. 
$$
On the other hand, notice that $Dw$ is the unique solution to \eqref{Eq-Dw}, we have 
$$
\langle \varphi, Dw_t\rangle _H= \int_t^T (L_r \langle \varphi,  Dw_r\rangle _H+\langle \varphi, g_r\rangle _H) \mathrm{d} r=\int_t^T (L_r \langle \varphi,  Dw_r\rangle _H+g^\varphi_r) \mathrm{d} r. 
$$
So $w^\varphi = \langle \varphi, Dw\rangle $, which implies $s\mapsto w^s$ is a version of $Dw$.  

{\em Step 4.} In this step, we prove the $C^{2+\alpha}$ regularity estimate for $v$. Define $u_t(x)={\mathbf E}^t w_t(x)$. Theorem \ref{Th-RPDE} and Lemma \ref{Le-con-norm} yield 
$$
\|u\|_{C^{2+\alpha}_{x,t}}\leqslant \|w\|_{C^{2+\alpha}_{x,t}} \leqslant C \|f\|_{C^{\alpha}_{x,t}}\leqslant C \Lambda_f. 
$$
Let $\dot w_t(x):= -[L_t w_t(x)+f_t(x)]$, by {\em Step 1}, $\dot w \in C^{\alpha,0}_{x,t}(Q_T; {\mathbb D}^{1,p})$. Note that 
$$
w_t(x)=w_0(x)+\int_0^t \dot w_s(x) \mathrm{d} s. 
$$
Thanks to Lemma \ref{Le-Ety}, for each $(x,t)\in Q_T$, 
$$
u_t(x)={\mathbf E}^t w_t(x)={\mathbf E} w_0(x) +\int_0^t {\mathbf E}^s\dot w_s(x)\mathrm{d} s + \int_0^t {\mathbf E}^s {\mathcal W}_{s,s}(x)\cdot \mathrm{d}  W_s, 
$$
where  ${\mathcal W}_{s,t}(x)=D_s w_0(x)+\int_0^t D_s\dot w_r (x)\mathrm{d} r$ for all $(x,t)\in Q_T$ and $s\in [0,T]$ a.e.. 
Since 
\begin{align*}
{\mathcal W}_{s,s}(x)=&D_s w_0(x)+\int_0^s D_s\dot w_r (x)\mathrm{d} r= \int_s^T D_s[L_r w_r+f_r] (x)\mathrm{d} r\\
=& \int_s^T[L_r D_s w_r +g^s_r](x)\mathrm{d} r=\int_s^T [L_r w^s_r +g^s_r](x)=w^s_s(x), \ s\in [0,T]\ a.e., 
\end{align*}
we get 
\begin{align*}
u_t(x)=& u_0(x) -\int_0^t {\mathbf E}^s(L_s w_s+f_s)(x) \mathrm{d} s + \int_0^t {\mathbf E}^s w^s_s(x)\cdot \mathrm{d}  W_s\\
=&u_0(x) -\int_0^t (L_s u_s+f_s)(x) \mathrm{d} s +\int_0^t {\mathbf E}^s  w^s_s(x)\cdot \mathrm{d}  W_s. 
\end{align*}
Note $u_T(x)=0$, we have 
$$
u_0(x)= \int_0^T (L_s u_s+f_s)(x) \mathrm{d} s - \int_0^T {\mathbf E}^s  w^s_s(x)\cdot \mathrm{d}  W_s. 
$$
Combining the above two equations, we obtain 
$$
u_t(x)=\int_t^T (L_s u_s+f_s)(x) \mathrm{d} s - \int_t^T {\mathbf E}^s  w^s_s(x)\mathrm{d} W_s 
$$
Let $v_s(x)=w^s_s(x)$, then the above identity implies $(u_t,v_t)=({\mathbf E}^tw_t, {\mathbf E}^tw^t_t)$ is a solution to \eqref{Eq-BSPDE}. Moreover,  
\begin{align*}
\|v\|_{C^{2+\alpha}_{x,t}}=&\sup_{0\leqslant t\leqslant T}\|{\mathbf E}^tw_t^t\|_{C^{2+\alpha}({\mathbb R}^n; L^p(\Omega))}\leqslant  \sup_{s\in [0,T]} \|w^s\|_{C^{2+\alpha}_{x,t}}\overset{\eqref{eq-ws-norm}}{\leqslant} C  \Lambda_f<\infty. 
\end{align*}
So we complete our proof. 
\end{proof}

Let $\varrho\in C_c^\infty({\mathbb R}^n)$ satisfying $\int \varrho =1$, and $\varrho_m (x):=m^n \varrho(m x)$. For any function $g: {\mathbb R}^n\to {\mathbb R}^m$, set $g^m:= g*\rho_m$. 

The following corollary of Theorem \ref{Th-BSPDE-Holder} is standard. 
\begin{corollary}[Stability]\label{Cor-St} 
Assume $a, b,c$ satisfy Assumption \ref{Aspt1} and \ref{Aspt2}. Let $w^m_t$ (respectively $(u^m, v^m)$) be the solution to \eqref{Eq-RPDE} (respectively \eqref{Eq-BSPDE}) in $C^{2+\alpha}_{x,t}$ (respectively $C^{2+\alpha}_{x,t}\times C^{2+\alpha}_{x,t}$) with $a, b,c ,f$ replaced by $a^m, b^m, c^m ,f^m$. Then for any $\beta\in(0,\alpha)$, it holds that 
\begin{align*}
\|\partial_t (w-w^m)\|_{C^\beta_{x,t}}+ \|w-w^m\|_{C^{2+\beta}_{x,t}}+ T^{-1}\|w-w^m\|_{C^{0}_{x,t}}\to 0\, (n\to \infty), 
\end{align*}
\begin{align*}
\|u-u^m\|_{C^{2+\beta}_{x,t}}+\|v-v^m\|_{C^{2+\beta}_{x,t}}\to 0 \, (n\to \infty).
\end{align*}
\end{corollary}

\section{SDEs with random singular coefficients}
In this section, we give the proof for our main result. 
\begin{proof}[Proof of Theorem \ref{Th-SDE-Main}]
We first point out that it is enough to  prove the well-posedness of \eqref{Eq-SDE} for $t\in [0,T/2]$, where $T$ is a universal constant depending only on $n, \alpha, \Lambda, p$. 

\noindent{\bf Pathwise uniqueness:} Assume $X_t$ is a solution to \eqref{Eq-SDE}. We prove the uniqueness by Zvonkin type transformation. With a little abuse of notation,  we denote $C^{\beta}_{x,t}=C^{\beta,0}_{x,t}(Q_T; L^p(\Omega; {\mathbb R}^m))$, where $m$ is an integer that can be changed in different places. Recalling that $L_t=a_t^{ij}\partial_{ij}+b^i_t\partial_i $. We consider the following BSPDE: 
 \begin{equation}\label{Bspde1}
\mathrm{d} u_t+(L_t u_t+b_t)\mathrm{d} t =v_t\cdot\mathrm{d} W_t, \quad u_T(x)=0. 
\end{equation}
By our assumptions and Theorem \ref{Th-BSPDE-Holder}, \eqref{Bspde1} has an ${\mathscr F}_t$-adapted solution $(u_t,v_t)$ and 
\begin{equation}\label{eq-u-v}
\|u\|_{C^{2+\alpha}_{x,t}} + \|v\|_{C^{2+\alpha}_{x,t}} <\infty. 
\end{equation}
Since $u_t={\mathbf E}^t w_t$, $w_t$ solves 
$$
\partial_t w+L_t w+b=0, \quad w_T(x)=0
$$
and 
$$
\mathrm{ess\,sup}_{\omega\in \Omega} \left(\sup_{t\in [0,T]}\|b_t(\cdot, \omega)\|_{C^\alpha}+\sup_{(s,t)\in\Delta_T} \|D_sb_t(\cdot, \omega)\|_{C^\alpha}\right )<\infty. 
$$
By Remark \ref{Rek-Classic}, we have 
$$
\begin{aligned}
&\mathrm{ess\,sup}_{\omega\in \Omega}  \sup_{t\in [0,T]}\left( \|w_t(\cdot, \omega)\|_{C^{2+\alpha}}+ T^{-1}\|w_t(\cdot, \omega)\|_{C^{\alpha}}\right )   \\
\leqslant & C \mathrm{ess\,sup}_{\omega\in \Omega} \sup_{t\in [0,T]} \|b_t(\cdot, \omega)\|_{C^\alpha}.  
\end{aligned}
$$
Interpolation inequality and above estimate yield 
$$
\mathrm{ess\,sup}_{\omega\in \Omega}\sup_{t\in [0,T]}\|u_t(\cdot ,\omega)\|_{C^1}\leqslant \mathrm{ess\,sup}_{\omega\in \Omega} \sup_{t\in [0,T]} \|w_t(\cdot, \omega)\|_{C^1}\leqslant C_T, 
$$
where $C_T\to0$ as $T\to0$. Below we fix $T=T(n,\alpha,\Lambda,p)>0$ so that 
$$
\mathrm{ess\,sup}_{\omega\in \Omega} \sup_{t\in [0,T]} \|u_t(\cdot, \omega)\|_{C^1}\leqslant \frac{1}{2}. 
$$
Let $\phi_t(x)=x+u_t(x)$, then
\begin{equation}\label{Eq-Phi-g}
\frac{1}{2}\leqslant \mathrm{ess\,sup}_{\omega\in \Omega}\sup_{0\leqslant t\leqslant T}\|\nabla \phi_t(x,\omega)\|_{L^\infty } \leqslant \frac{3}{2}. 
\end{equation}
So for almost surely $\omega\in \Omega$, $\phi_t(\cdot, \omega)$ is a stochastic $C^{2+\alpha}$-differential homeomorphism from ${\mathbb R}^n$ to ${\mathbb R}^n$. 
By the definition of $\phi$, 
$$
\mathrm{d} \phi_t(x)= -(L_t u_t(x)+b_t(x)) \mathrm{d} t+ v_t(x)\cdot \mathrm{d}  W_t=\mathrm{d} u_t(x)  = \mathrm{d} g_t(x)+\mathrm{d} m_t(x), 
$$
where  
\begin{equation}\label{def-g-m}
g_t(x)= -\int_0^t (L_s u_s(x)+b_s(x)) \mathrm{d} s, \quad m_t(x):= \int_0^t v_s(x) \mathrm{d} W_s. 
\end{equation}
We want to show that $\phi, g, u, v$ and $X$ are regular enough to apply the It\^o-Wenzell formula (see Lemma \ref{Le-Ito}). Since $\|v\|_{C^{2+\alpha}_{x,t}}<\infty$,  we have 
$$
\sup_{t\in [0,T]; x \neq y} \frac{{\mathbf E} |\nabla^2v_t(x)-\nabla^2v_t(y)|^p}{|x-y|^{\alpha p}}  <\infty. 
$$
Note that $p>n/\alpha$, so for any $\beta\in (n/p, \alpha)$ and $N>0$, by Garsia-Rademich-Rumsey's inequality, 
\begin{align*}
&\sup_{t\in [0,T]}{\mathbf E} \left(\sup_{x,y\in B_N} \frac{|\nabla v_t^2(x)-\nabla^2 v_t(y)|}{|x-y|^{\beta-n/p}} \right )^p \\
\leqslant &C_N  \sup_{t\in [0,T]} {\mathbf E} \left( \int_{B_N}\int_{B_N} \frac{|\nabla^2 v_t(x)-\nabla^2v_t(y)|^{p}}{|x-y|^{d+\beta p}} \mathrm{d} x\mathrm{d} y\right )\\
\leqslant &C_N \int_{B_N}\int_{B_N}|x-y|^{-d+(\alpha-\beta)p} \leqslant C_N. 
\end{align*}
Combining this and the fact that $\sup_{t\in [0,T]} {\mathbf E} |\nabla^2 v_t(0)|^p<\infty$, we get 
\begin{align*}
\sup_{t\in[0,T]} {\mathbf E} \left(\sup_{x\in B_N} |\nabla^2 v_t(x)|^p \right )<\infty, \quad \forall N>0. 
\end{align*}
Moreover, one can also prove  
\begin{equation}\label{Eq-v-C2}
\sup_{t\in[0,T]} {\mathbf E} \|v_t\|_{C^{2}(B_N)}^p<\infty , \quad \forall N>0. 
\end{equation}
Recalling that $g_t(x)$ and $m_t(x)$ are defined in \eqref{def-g-m}, let 
$$
\eta_t(x):=\int_0^t g_s(x) \mathrm{d} s \overset{\eqref{Bspde1}}{=}  u_t(x)-u_0(x)-m_t(x). 
$$ 
By Burkholder-Davis-Gundy's inequality, for each $k=0,1,2$
\begin{align*}
&{\mathbf E} \left|\nabla^k m_t(x) - \nabla^k m_t(y) \right|^p\leqslant C{\mathbf E} \left[\int_0^t |\nabla^k v_s(x)-\nabla^k v_s(y)|^k \mathrm{d} s\right ]^{\frac{p}{2}}\\
\leqslant& C{\mathbf E} \int_0^t |\nabla^k v_s(x)-\nabla^k v_s(y)|^p \mathrm{d} s\leqslant C |x-y|^{\alpha p} \|\nabla^k v\|_{C^\alpha_{x,t}}^p, 
\end{align*}
which together with \eqref{eq-u-v} implies   
\begin{align*}
\| \eta\|_{C^{2+\alpha}_{x,t}} \leqslant C\left(\|u\|_{C^{2+\alpha}_{x,t}}+ \|v\|_{C^{2+\alpha}_{x,t}}\right ). 
\end{align*}
By the definition of $\eta$, 
$$
\|\partial_t \eta\|_{C^\alpha_{x,t}}=\|g\|_{C^\alpha_{x,t}}\leqslant \|L_t u+b\|_{C^\alpha_{x,t}}\leqslant C\left( \|u\|_{C^{2+\alpha}_{x,t}}+ \|b\|_{C^\alpha_{x,t}}\right ) . 
$$
Thanks to Lemma \ref{Le-Inter}, for any $\beta\in (n/p, \alpha)$ and $\theta=\frac{1}{2}+\frac{\alpha-\beta}{2}\in (\frac{1}{2}, 1)$, we have 
$$
\|\eta\|_{C^{\theta}_tC^{1+\beta}_x} \leqslant C \|\partial_t\eta\|_{C^\alpha_{x,t}}^\theta \|\eta\|_{C^{2+\alpha}_{x,t}}^{1-\theta}. 
$$
By the same procedure of proving  \eqref{Eq-v-C2}, we have 
\begin{align*}
\left[ {\mathbf E}  \left\|\int_{t_1}^{t_2}g_s\mathrm{d} s\right\|_{C^1(B_N)}^p\right ]^{1/p} =&\left[ {\mathbf E}  \|\eta_{t_1}-\eta_{t_2}\|_{C^1(B_N)}^p\right ]^{1/p} \\
\leqslant &C_N |t_1-t_2|^{\theta}, \quad \theta\in (1/2,1). 
\end{align*}
On the other hand, ${\mathbf E} |X_{t_1}-X_{t_2}|^{p'}\leqslant C |t_1-t_2|^{\frac{p'}{2}}(p'=p/(p-1))$. 
So $\phi, g, v, X$ satisfy all the conditions in Lemma \ref{Le-Ito}. Using \eqref{Eq-IW}, we get 
\begin{align*}
\begin{aligned}
\mathrm{d} \phi_t(X_t)=&-L_t u_t(X_t)-b_t(X_t)\mathrm{d} t+v^k_t(X_t)  \mathrm{d}  W_t^k \\
&+[b^i_t(X_t)  \partial_i \phi_t(X_t) + a^{ij}_t(X_t) \partial_{ij} \phi_t(X_t)+\partial_{i} v^k_{t}(X_t)\sigma_{t}^{i k}(X_t)]\mathrm{d} t\\
&+\partial_i \phi_t(X_t) \sigma^{ik}_t(X_t) \mathrm{d}  W_t^k\\
=&\partial_i v^k_t(X_t)\sigma_t^{ik}(X_t)\mathrm{d} t+ \partial_i \phi_t(X_t)\sigma_t^{ik}(X_t)\mathrm{d} W^k_t+v_t^k(X_t)\mathrm{d}  W_t^k. 
\end{aligned}
\end{align*}
Set 
$$
Y_t =\phi_t(X_t), \quad \widetilde {b}_t(y)=\partial_iv^k_t\sigma_t^{ik}\circ \phi_t^{-1}(y)\ \mbox{ and }\  \widetilde {\sigma}_t(y)=[\nabla \phi_t \sigma_t+v_t]\circ\phi^{-1}_t(y). 
$$ 
By the above calculations, one sees that 
\begin{equation}\label{Eq-SDE-Y}
Y_t=Y_0+ \int_0^t \widetilde {b}_s(Y_s) \mathrm{d} s+\int_0^t \widetilde {\sigma}_s(Y_s)\mathrm{d}  W_s. 
\end{equation}
Thanks to Lemma  \ref{Le-measurability}, 
$\widetilde  b$ and $\widetilde  \sigma$ are ${\mathscr B}\times {\mathscr P}$-measurable. For any $x, y\in B_N$ and $t\in [0,T]$, by the definitions of $\widetilde b$ and $\widetilde \sigma$, we have 
$$
|\widetilde  b_t(0)|+|\widetilde  \sigma_t(0)|\leqslant C K_t^N, 
$$
$$
|\widetilde  b_t(x)-\widetilde  b_t(y)|+|\widetilde  \sigma_t(x)-\widetilde  \sigma_t(y)| \leqslant C K_t^N |x-y|,
$$
where $K^N_t:=  \|u_t\|_{C^2(B_N)}+\|v_t\|_{C^2(B_N)}$. It is not hard to see that $K_t^N$ is progressive measurable and satisfies 
$$
{\mathbf E} \int_0^T K_t^N\mathrm{d} t  \leqslant T \sup_{t\in [0,T]} {\mathbf E} K_t^N \overset{\eqref{Eq-v-C2}}{<}\infty. 
$$
Thanks to Theorem 1.2 of \cite{krylov1999kolmogorov}, equation \eqref{Eq-SDE-Y} admits a unique solution, which implies $X_t$ is unique up to indistinguishability. 

\medskip

\noindent {\bf Existence:} 
Let $b^m_t=b_t*\varrho_m$ and $X^m$ be the solution to 
\begin{align}\label{Eq-Xn}
X^m_t=X_0+ \int_0^t b_s^m(X^m_s) \mathrm{d} s+ \int_0^t \sigma_s(X^m_s)\mathrm{d}  W_s, \ t\in [0,T].  
\end{align}
We claim that $X^m_t$ uniform convergence on compacts in probability (ucp convergence in short) to a process $X_t$. Let $(u^m,v^m)$ be the pair of functions constructed in Theorem \ref{Th-BSPDE-Holder} satisfying  
$$
\mathrm{d} u^m_t+ \left[a_t^{ij}\partial_{ij} u^m_t+(b^m_t)^i\partial_i u_t^m+b^m_t\right ]\mathrm{d} t=v^m_t \cdot \mathrm{d} W_t. 
$$
Like before, we can find a uniform  constant $T=T(n,\alpha,\Lambda,p)>0$ such that $\|\nabla u^m_t\|_{L^\infty}\leqslant 1/2$. Define  $\phi^m_t(x):=x+u^m_t(x), \ Y^m_t:= \phi^m_t(X^m_t)$ and $Z^{m,m'}_t:= Y_t^m-Y^{m'}_t$.  Again by It\^o-Wentzell's formula, we have 
\begin{align*}
Z^{m,m'}_t=&Y^m_t-Y^{m'}_t= u^m_0(X_0)-u^{m'}_0(X_0)+\int_0^t [\widetilde  b^m_s(X^{m}_s)-\widetilde  b^{m'}_s(X^{m'}_s)] \mathrm{d} s\\
&+ \int_0^t [\widetilde  \sigma^m_s(X^{m}_s)-\widetilde  \sigma^{m'}_s(X^{m'}_s)] \mathrm{d} W_s, 
\end{align*}
where 
$$
\widetilde  b^m_t:=[\partial_i v^{m,k}_t\sigma_t^{ik}]\circ(\phi_t^m)^{-1},\quad \widetilde \sigma^m_t:=[(\nabla\phi^m_t) \sigma_t+ v_t^m] \circ(\phi_t^m)^{-1}. 
$$
By It\^o's formula, for any stopping time $\tau\leqslant T$, 
\begin{align}\label{Eq-Z2}
\begin{aligned}
\left|Z^{m,m'}_{t\wedge \tau}\right|^2=& |u^m_0(X_0)-u^{m'}_0(X_0)|^2+ 2 \int_0^{t\wedge \tau} Z^{m,m'}_s\cdot \left[\widetilde  b^m_s(Y^m_s)-\widetilde  b^{m'}_s (Y^{m'}_s)\right ] \mathrm{d} s\\
&+\int_0^{t\wedge \tau} \mathrm{tr}\left[\widetilde  \sigma_s^m(Y^m_s)-\widetilde  \sigma_s^{m'}(Y^{m'}_s)\right ]\left[\widetilde  \sigma_s^m(Y^{m}_s)-\widetilde  \sigma_s^{m'}(Y^{m'}_s)\right ]^* \mathrm{d} s +m_{t\wedge \tau}, 
\end{aligned}
\end{align}
where 
$$
m_t= 2\int_0^{t} Z^{m,m'}_s\cdot \left[\widetilde  \sigma^m_s(Y^m_s)-\widetilde  \sigma^{m'}_s (Y^{m'}_s)\right ] \mathrm{d}  W_s.
$$ 
For any $N, k\in {\mathbb N}$, let $K^{m,N}_t:=  \|u^m_t\|_{C^2(B_N)}+\|v^m_t\|_{C^2(B_N)}$, 
$$
\tau^{N, k}= \inf_{m}\, \inf\left\{t\geqslant 0:  \int_0^t (K_s^{^m, N})^2 \mathrm{d} s\geqslant  k\right\}\wedge T, 
$$
and 
$$
\sigma^{N} = \inf_{m}\,\inf\left\{t\geqslant 0:  |Y^m_t|>N/2\right\}\wedge T, \quad \sigma^{N, k}:= \sigma^N\wedge \tau^{N, k}. 
$$
For all $x, y\in B_{N/2}$ and $t\in [0, \sigma^{N,k}]$, we have 
\begin{equation}\label{eq-lip-ab}
\sup_{m\in {\mathbb N}} \left(|\widetilde b_t^m(x)-\widetilde b_t^m (y)|+|\widetilde \sigma_t^m(x)-\widetilde \sigma_t^m (y)|\right ) \leqslant C_k |x-y|. 
\end{equation}
Since for each $(x,t)\in B_{N/2}\times [0,T]$, $(\phi_t^{m})^{-1}(x)\in B_{N}$, we obtain that for any $x\in B_{N/2}$ and $t\in [0,\tau^{N,k}]$, 
\begin{align*}
&|\widetilde b_t^m(x)-\widetilde b_t^{m'}(x)| \\
\leqslant& \left|[\partial_i v^{m,k}_t\sigma_t^{ik}]\circ(\phi_t^m)^{-1}(x)-[\partial_i v^{m',k}_t\sigma_t^{ik}]\circ(\phi_t^{m})^{-1}(x)\right|\\
& + \left|[\partial_i v^{m',k}_t\sigma_t^{ik}]\circ(\phi_t^{m})^{-1}(x)-[\partial_i v^{m',k}_t\sigma_t^{ik}]\circ(\phi_t^{m'})^{-1}(x)\right|\\
\leqslant & C \|\nabla v^m_t-\nabla v^{m'}_t\|_{L^\infty(B_{N})} + C \| v^{m'}_t\|_{C^2(B_N)} |(\phi_t^{m})^{-1}(x)-(\phi_t^{m'})^{-1}(x)|\\
\leqslant & C \|v^m_t- v^{m'}_t\|_{C^2(B_{N})} + C \| v^{m'}_t\|_{C^2(B_N)} \sup_{y\in B_N}|\phi_t^{m'}(y)-\phi_t^m(y)|\\
\leqslant &C_k \left( \|u^m_t-u_t^{m'}\|_{C^2(B_N)}+\|v^m_t-v^{m'}_t\|_{C^2 (B_N)}\right ). 
\end{align*}
Similarly, for each $x\in B_{N/2}$ and $t\in [0,\tau^{N,k}]$, 
\begin{align*}
|\widetilde \sigma^m_t(x)-\widetilde \sigma^{m'}_t(x)|\leqslant C_k \left( \|u^m_t-u_t^{m'}\|_{C^2(B_N)}+\|v^m_t-v^{m'}_t\|_{C^2 (B_N)}\right )
\end{align*}
By our Theorem \ref{Th-BSPDE-Holder},  Corollary \ref{Cor-St} and the same procedure of proving \eqref{Eq-v-C2}, we have 
$$
\sup_{t\in [0,T], m\in{\mathbb N}}{\mathbf E} |K^{m,N}_t|^p=\sup_{t\in [0,T], m\in{\mathbb N}} {\mathbf E} \left(\|u^m_t\|_{C^2(B_N)}+\|v^m_t\|_{C^2(B_N)}\right )^p<\infty
$$
and 
$$
\lim_{m\to\infty} \sup_{t\in [0,T]} {\mathbf E}\left( \|u_t-u^m_t\|_{C^2(B_N)}+ \|v_t-v^m_t\|_{C^2(B_N)}\right )^p=0. 
$$
Thus,  
\begin{equation}\label{eq-lim-tau-si}
\lim_{k\to\infty} \tau^{N, k}=T, \quad \lim_{N\to\infty} \sigma^N=T
\end{equation}
and 
\begin{equation}\label{eq-m-m}
\lim_{m,m'\to\infty}{\mathbf E} \left(\|\widetilde \sigma_t^m-\widetilde \sigma_t^{m'}\|_{L^\infty(B_{N/2})}^p+\|\widetilde b_t^m-\widetilde b_t^{m'}\|_{L^\infty(B_{N/2})}^p {\mathbf{1}}_{[0,\sigma^{N,k}]}(t) \right )=0. 
\end{equation}
Let $\tau=\sigma^{N, k}$ in \eqref{Eq-Z2}. Using \eqref{eq-lip-ab}, we have 
\begin{align*}
\left|Z^{m,m'}_{t\wedge \sigma^{N, k}}\right|^2\overset{\eqref{eq-lip-ab}}{\leqslant} & \left|u^m_0(X_0)-u^{m'}_0(X_0)\right|^2\\
&+ C_k \int_0^{t\wedge \sigma^{N, k}} 
\left|Z^{m,m'}_s\right| \left(\left|Z^{m,m'}_s\right|+\| \widetilde  b^m_s-\widetilde  b^{m'}_s \|_{L^\infty(B_{N/2})} \right ) \mathrm{d} s\\
&+ C_k \int_0^{t\wedge \sigma^{N, k}} 
\left( \left|Z^{m,m'}_s\right|+\| \widetilde  \sigma^m_s-\widetilde  \sigma^{m'}_s \|_{L^\infty(B_{N/2})} \right )^2 \mathrm{d} s+ m_{t\wedge \sigma^{N,k}}\\
\leqslant & \|u^m_0-u^{m'}_0\|_{L^\infty}^2 +C_k\int_0^{t } \left|Z^{m,m'}_{s\wedge \sigma^{N,k}}\right|^2 \mathrm{d} s + m_{t\wedge \sigma^{N,k}}\\
&+ C_k\int_0^{t\wedge \sigma^{N,k}} \left(\|\widetilde  b^m_s-\widetilde  b^{m'}_s\|_{L^\infty(B_{N/2})}^2+\|\widetilde  \sigma^m_s-\widetilde  \sigma^{m'}_s\|_{L^\infty(B_{N/2})}^2 \right )\mathrm{d} s. 
\end{align*}
By Gronwall's inequality and \eqref{eq-m-m}, we get 
\begin{equation}\label{eq-Zm-Zm'}
\begin{aligned}
{\mathbf E} &\left|(Z^{m,m'})^*_{T\wedge \sigma^{N, k}}\right|^2\leqslant C_k\|u^m_0-u^{m'}_0\|_{L^\infty}^2\\
&+  C_k  {\mathbf E}\int_0^T \left(\|\widetilde  \sigma^m_s-\widetilde  \sigma^{m'}_s\|_{L^\infty(B_{N/2})}^2+\|\widetilde  b^m_s-\widetilde  b^{m'}_s\|_{L^\infty(B_{N/2})}^2  \right ){\mathbf{1}}_{[0,\sigma^{N,k}]}(s) \mathrm{d} s \\
& \overset{\eqref{eq-m-m}}{\longrightarrow} 0 \ (m,m'\to \infty). 
\end{aligned}
\end{equation}
On the other hand, 
\begin{equation}\label{eq-Xm-Xm'}
\begin{aligned}
|X^m_t-X^{m'}_t|=& |(\phi_t^m)^{-1}(\phi_t^m(X^m_t))-(\phi_t^m)^{-1}(\phi_t^m(X^{m'}_t))|\leqslant 2 |\phi_t^m(X_t^m)- \phi_t^m(X_t^{m'})| \\
\leqslant &  2|\phi^m_t(X_t^m)- \phi^{m'}_t(X_t^{m'})|+ 2| \phi^{m'}_t(X_t^{m'})- \phi_t^m(X_t^{m'})|\\
\leqslant & 2 \|u_t^m-u^{m'}_t\|_{L^\infty}+2|Y^m_t-Y^{m'}_t|. 
\end{aligned}
\end{equation}
Combining \eqref{eq-Zm-Zm'} and \eqref{eq-Xm-Xm'},   we get 
\begin{align*}
&\lim_{m,m'\to\infty}{\mathbf E} \sup_{t\in[0,T]}|X_{t\wedge \sigma^{N, k} }^m-X_{t\wedge \sigma^{N, k}}^{m'}| \\
\leqslant & 2 \lim_{m,m'\to\infty}  {\mathbf E} \sup_{t\in [0,T]}\|u_t^m-u^{m'}_t\|_{L^\infty} +2 \lim_{m,m'\to\infty} {\mathbf E} (Z^{m,m'})^*_{T\wedge \sigma^{N,k}}=0. 
\end{align*}
Noting 
$$
\lim_{N\to\infty} \lim_{k\to \infty} \sigma^{N,k}=\lim_{N\to\infty}\sigma^N\overset{\eqref{eq-lim-tau-si}}{=} T,  
$$
we obtain 
\begin{equation}\label{eq-Xmm'}
\lim_{m,m'\to\infty}  {\mathbf P}\left( \sup_{t\in [0,T/2]} |X_t^m-X_t^{m'}|>\varepsilon)\right )=0,  \quad \forall \varepsilon>0. 
\end{equation}
This implies that there is a continuous process $\{X_t\}_{t\in [0,T/2]}$ such that $X^m\to X$ in the sense of ucp. Hence, 
$$
\int_0^t \sigma_s(X_s^m)\cdot \mathrm{d}  W_s\overset{{\mathbf P}} {\longrightarrow} \int_0^t \sigma_s(X_s) \mathrm{d} W_s, \quad \forall t\in [0,T/2], 
$$
and for each $t\in [0,T/2]$ and $\varepsilon>0$, 
\begin{align*}
&{\mathbf P}\left( \left|\int_0^t b^m_s(X^m_s)\mathrm{d} s-\int_0^tb_s(X_s)\mathrm{d} s\right| >\varepsilon\right )\\
\leqslant& {\mathbf P}\left( \sup_{t\in [0,T/2]}|b^m(X_t^m)-b(X_t^m)|>\frac{\varepsilon}{2}\right )+ {\mathbf P}\left( \sup_{t\in [0,T/2]}|b_t(X^m_t)-b_t(X_t)|\mathrm{d} s  >\frac{\varepsilon}{2}\right )\\
\leqslant& {\mathbf P}\left( \|b^m-b\|_{L^\infty(Q_T)} >\frac{\varepsilon}{2}\right )+ {\mathbf P}\left( \sup_{t\in [0,T/2]}|X_t^m-X_t|^\alpha > \frac{\varepsilon}{2\Lambda} \right )\to 0, \quad (m\to \infty). 
\end{align*}
Taking limit on both side of \eqref{Eq-Xn}, one sees that $X$ is a solution to \eqref{Eq-SDE}.
\end{proof}

\section{appendix}
In this section, we give some Lemmas used in the previous sections. The following basic result is useful. 
\begin{lemma}\label{Le-Chart-Holder}
Let $f\in L^1({\mathbb R}^n; {\mathcal B})+ L^\infty({\mathbb R}^n; {\mathcal B})$. 
\begin{enumerate}
\item (Bernstein's inequality) 
For any $k=0,1,2,\cdots$, there is a constant $C=C(n,k)>0$ such that for all $j=-1,0,1,\cdots$,
\begin{align}\label{B1}
\|\nabla^k\Delta_j f\|_0 \leqslant C2^{k j}\|\Delta_jf\|_0; 
\end{align}
\item For any $\alpha\in (0,1)$, there is a constant $C=C(\alpha, n)>1$ such that 
\begin{align}\label{Ch}
C^{-1}\sup_{j\geqslant -1} 2^{j\alpha} \|\Delta_j f\|_0 \leqslant \| f\|_{C^\alpha} \leqslant C\sup_{j\geqslant -1} 2^{j\alpha} \|\Delta_j f\|_0. 
\end{align}
\end{enumerate}
\end{lemma}
One can find the proof of above lemma in \cite{bahouri2011fourier} for ${\mathcal B}={\mathbb R}$.  We present its Banach-valued version below for the reader's convenience.
\begin{proof}
For any $j=0,1,2,\cdots$, we have $\int_{{\mathbb R}^n} h_j(z) \mathrm{d} z= \varphi_j(0)=0$, so 
\begin{align*}
\|\Delta_j f(x)\|_{{\mathcal B}} = &\left\|\int_{{\mathbb R}^n} h_j(x-y) [f(y)-f(x)] \mathrm{d} y \right\|_{{\mathcal B}} \\
=& \left\|\int_{{\mathbb R}^n} 2^{jn}h(2^j(x-y)) [f(y)-f(x)] \mathrm{d} y \right\|_{{\mathcal B}} \\
\leqslant& C\|f\|_{C^\alpha} \int_{{\mathbb R}^n} 2^{jn} h(2^j z)|z|^\alpha \mathrm{d} z  =C2^{-j\alpha} \|f\|_{C^\alpha}, 
\end{align*}
which implies 
$$
\sup_{j\geqslant -1} 2^{j\alpha} \|\Delta_j f\|_0 \leqslant C_\alpha\|f\|_{C^\alpha}. 
$$
On the other hand, 
\begin{align*}
&\Big\|f(x)-\sum_{j=-1}^k \Delta_j f(x)\Big\|_{{\mathcal B}} = \Big\| \int_{{\mathbb R}^n} {\mathscr F}^{-1} (\chi(2^{-k}\cdot)) (y) [f(x)-f(x-y)]\mathrm{d} y\Big\|_{{\mathcal B}}\\
=& \Big\|\int_{B_{2^k\varepsilon}} {\mathscr F}^{-1}(\chi) (z)[f(x)-f(x-2^{-k} z)] \mathrm{d} z  \Big\|_{{\mathcal B}}\\
&+\Big\|\int_{B^c_{2^k\varepsilon}} {\mathscr F}^{-1}(\chi) (z)[f(x)-f(x-2^{-k} z)] \mathrm{d} z  \Big\|_{{\mathcal B}}\\
\leqslant &  \mathrm{osc}_{B_\varepsilon(x)}f \cdot \int_{{\mathbb R}^n} |{\mathscr F}^{-1}(\chi) (y)| \mathrm{d} y+ 2 \|f\|_0 \int_{B_{2^k\varepsilon}^c} |{\mathscr F}^{-1}(\chi)(y)| \mathrm{d} y.  
\end{align*}
Let $k\to \infty$ and then $\varepsilon\to 0$, we obtain that for each $f\in C_b({\mathbb R}^n; {\mathcal B})$ and $x\in {\mathbb R}^n$, 
$f(x)=\sum_{j\geqslant -1} \Delta_j f(x)$. Thus, for any $K> 0$, 
\begin{align*}
|f(x)-f(y)|_{{\mathcal B}} \leqslant & \sum_{j\geqslant -1} |\Delta_j f(x)-\Delta_j f(y)|_{{\mathcal B}} \leqslant |x-y|\sum_{-1\leqslant j\leqslant K}\|\nabla \Delta_j f\|_0+ 2\sum_{j>K} \|\Delta_j f\|_0\\
\leqslant& C_\alpha (|x-y| 2^{(1-\alpha)K} + C2^{-\alpha K} )\sup_{j\geqslant -1} 2^{\alpha j}\|\Delta_j f\|_{0}. 
\end{align*}
For any $|x-y|<1$, by choosing  $K=-\log_2(|x-y|)$, we obtain 
$$
|f(x)-f(y)|_{{\mathcal B}}\leqslant C_\alpha |x-y|^\alpha \sup_{j\geqslant -1} 2^{\alpha j}\|\Delta_j f\|_{0}. 
$$
So we complete our proof. 

\end{proof}
Suppose $f:{\mathbb R}^n\to {\mathbb R}^n$ is a continuous homeomorphism on ${\mathbb R}^n$, its inverse map is denoted by $f^{-1}$. Our next auxiliary lemma is used in the proof of Theorem \ref{Th-SDE-Main}. 
\begin{lemma}\label{Le-measurability}
Suppose $(S; {\mathcal S})$ is a measurable space, $F: (S\times {\mathbb R}^n; {\mathcal S}\times {\mathscr B})\to ({\mathbb R}^n; {\mathscr B})$. 
\begin{enumerate}
\item 
Assume $X$ is another measurable map from $(S; {\mathcal S})$ to $({\mathbb R}^n; {\mathscr B})$.  Then the map $a\mapsto F(a, X(a))$ is measurable from $(S; {\mathcal S})$ to $({\mathbb R}^n; {\mathscr B})$. 

\item 
For any $L>0$ define  
\begin{align*}
H_L:=\big\{f: &{\mathbb R}^n\to {\mathbb R}^n| f\ \mbox{is a continuous homeomorphism and}\\
&L^{-1}|x-y|\leqslant |f(x)-f(y)|\leqslant L|x-y|\big\}. 
\end{align*}
If $F: (S\times {\mathbb R}^n; {\mathcal S}\times {\mathscr B})\to ({\mathbb R}^n; {\mathscr B})$ and for each $a\in S$, $F(a,\cdot)\in H_L$, then 
 the map 
$$
F^{-1}: S\times {\mathbb R}^n\ni(a,x)\mapsto [F^{-1}(a,\cdot)](x) \in {\mathbb R}^n
$$ 
is ${\mathcal S}\times {\mathscr B}/ {\mathscr B}$ measurable. 
\end{enumerate}
\end{lemma}
\begin{proof}
(1) This conclusion is trivial since the map $a\mapsto (a, X(a))$ is ${\mathcal S}/ {\mathcal S}\times {\mathscr B}$ measurable. 

\noindent(2). 
 Define 
$$
d(f,g):= \sup_{x\in {\mathbb R}^n} \frac{|f(x)-g(x)|}{1+|x|}, \quad \forall f, g\in H_L. 
$$
It is easy to verify that $H_L$ is a metric space equipped with metric $d$. For any $f\in H_L$ and $\varepsilon>0$, by the continuity of $x\mapsto F(a, x)$, we get 
\begin{align*}
&\{a: d( F(a,\cdot), f)<\varepsilon \}\\
=&\bigcap_{\substack{q\in {\mathbb Q}^n; \\r\in{\mathbb Q}\cap[0,1)}} \left\{a: \frac{|F(a,q)-f(q)|}{ 1+|q|}<r\varepsilon\right\}\in {\mathcal S}. 
\end{align*}
So the map $\overline F: (S, {\mathcal S})\to (H_L, {\mathscr B}(H_L;d))$ is measurable.  Obviously, the map 
$$\mathrm{Inv}: H_L\ni f\mapsto f^{-1}\in H_L,  
$$ 
is well-defined. Now assume $d(f_n,f)\to 0$. Given $x\in {\mathbb R}^n$, assume $y=f^{-1}(x)$, then 
\begin{align*}
|f_n^{-1}(x)-f^{-1}(x)|=& |f^{-1}_n\circ f(y)-f_n^{-1}\circ f_n(y)|\\
\leqslant & L|f(y)-f_n(y)| \leqslant  L (1+|y|) d(f_n, f). 
\end{align*}
By definition of $H_L$, 
$$
|f(y)-f(0)|\geqslant L^{-1} |y|, 
$$
which implies 
$$
|x|=|f(y)|\geqslant L^{-1} |y|-|f(0)|.  
$$
So
\begin{align*}
|f_n^{-1}(x)-f^{-1}(x)|\leqslant& L (1+ Lf(0)+L|x|) d(f_n, f)\\
\leqslant & C_{f, L} (1+|x|) d(f_n,f), 
\end{align*}
which implies $d (f_n^{-1},  f^{-1})\leqslant C_{f, L} d(f_n,f)\to 0$. Thus, the map $\mathrm{Inv}: H_L\to H_L$ is continuous. Hence, the map  $\overline F^{-1}: =\mathrm{Inv}\circ \overline F$ from $(S, {\mathcal S})$ to $(H_L, {\mathscr B}(H_L))$ is also measurable. As a consequence, the map 
$$
F^{-1}: S\times {\mathbb R}^n\ni(a,x)\mapsto [\mathrm{Inv}\circ F(a,\cdot)](x) \in {\mathbb R}^n
$$ 
is ${\mathcal S}\times {\mathscr B}/ {\mathscr B}$ measurable. 
\end{proof}
Roughly speaking, the above lemma shows that if $(a,x)\mapsto F(a,x)$ is measurable then $(a,x)\mapsto F^{-1}(a,\cdot)(x)$ is also measurable. 

\medskip

The following interpolation lemma is used several times in our paper. 
\begin{lemma}\label{Le-Inter}
Let $0\leqslant\gamma_0<\gamma_1<\gamma_2$ with $\gamma_1\notin{\mathbb N}$ and 
$\theta:=(\gamma_2-\gamma_1)/(\gamma_2-\gamma_0)\in (0,1)$,   $Q_T={\mathbb R}^n\times [0,T]$ and  ${\mathcal B}$ be a Banach space. Then there is a constant $C>0$, such that for all $f\in C^{\gamma_2}_{x,t}$ with $\partial_t f\in C^{\gamma_0}_{x,t}$,
\begin{align}\label{EX1}
\|f_{t_1}-f_{t_2}\|_{C^{\gamma_1}}\leqslant C |t_1-t_2|^\theta
\|\partial_t f\|^\theta_{C^{\gamma_0}_{x,t}} \|f\|^{1-\theta}_{C^{\gamma_2}_{x,t}}.
\end{align}
\end{lemma}
\begin{proof}
	First of all, for any $t\in [0,1]$, we have 
	\begin{align*}
	\|f_t\|_{C^{\gamma_1}}\leqslant C \|f_t\|_{C^{\gamma_0}}^\theta\|f_t\|_{C^{\gamma_2}}^{1-\theta}.
	\end{align*}
	For any $0\leqslant t_0<t_1\leqslant T$, $\beta\in(0,\theta)$ and $q>1/\theta$, by Garsia-Rademich-Rumsey's inequality, we have
	\begin{align*}
	\frac{\|f_{t_1}-f_{t_0}\|_{C^{\gamma_1}}^q}{|t_1-t_0|^{\beta q-1}}\leqslant&  C \int^{t_1}_{t_0}\!\!\int^{t_1}_{t_0}
	\frac{\|f_t-f_s\|^q_{C^{\gamma_1}}}{|t-s|^{1+\beta q}}\mathrm{d} s \mathrm{d} t\\
	\leqslant& C \int^{t_1}_{t_0}\!\!\int^{t_1}_{t_0} 
	\|f_t-f_s\|^{\theta q}_{C^{\gamma_0}}
	\|f_t-f_s\|^{(1-\theta)q}_{C^{\gamma_2}}|t-s|^{-1-\beta q}\mathrm{d} s\mathrm{d} t\\
	\leqslant& C\left(\int^{t_1}_{t_0}\!\!\int^{t_1}_{t_0}\frac{|t-s|^{\theta q}}{|t-s|^{1+\beta q}}
	\mathrm{d} s\mathrm{d} t\right)\|\partial_t f\|^{\theta q}_{C^{\gamma_0}_{x,t}}
	\|u\|^{(1-\theta) q}_{C^{\gamma_2}_{x,t}}\\
	=&C  |t_1-t_0|^{\theta q-\beta q+1} \|\partial_t f\|^{\theta q}_{C^{\gamma_0}_{x,t}}
	\|u\|^{(1-\theta) q}_{C^{\gamma_2}_{x,t}}, 
	\end{align*}
	which gives \eqref{EX1}. 
\end{proof}

\begin{lemma}\label{Le-con-norm}
Suppose $\beta\geqslant 0$,  ${\mathcal H}$ is a real Hilbert space and  $C^\beta_x=C^\beta({\mathbb R}^n; L^p(\Omega;{\mathcal H}))$. Assume ${\mathscr G}$ is a subalgebra of ${\mathscr F}$, then 
\begin{equation}\label{Eq-Enorm}
\|{\mathbf E} (X| {{\mathscr G}} )\|_{C^{\beta}_x}\leqslant \| X\|_{C^{\beta}_x}. 
\end{equation}
Moreover, for any $k\in {\mathbb N}$ with $k\leqslant \beta$, 
\begin{equation}\label{Eq-DE}
\nabla^k {\mathbf E}\left(X(x)|{\mathscr G}\right )= {\mathbf E}\left( \nabla^k  X(x)|{\mathscr G}\right ). 
\end{equation}
\end{lemma}
\begin{proof}
We only prove \eqref{Eq-Enorm} when $\beta\in (0,1)$. Denote $ {\mathbf E}^{\mathscr G}  X(\cdot):= {\mathbf E} (X(\cdot)| {{\mathscr G}} )$, by Jensen's inequality, 
\begin{align*}
&{\mathbf E} \left|{\mathbf E}^{\mathscr G}  X(x)-{\mathbf E}^{\mathscr G}  X(y)\right|_{{\mathcal H}}^p
\leqslant {\mathbf E} [{\mathbf E}^{\mathscr G}  |X(x)-X(y)|_{{\mathcal H}}]^p \\
\leqslant& {\mathbf E} [{\mathbf E}^{\mathscr G}  \left|X(x)-X(y)\right|_{{\mathcal H}}^p] = {\mathbf E}   \left|X(x)-X(y)\right|_{{\mathcal H}}^p \leqslant |x-y|^{\beta p} \|X\|_{C^\beta_x}^p, 
\end{align*}
which yields 
$$
\|{\mathbf E}^{\mathscr G} X\|_{C^\beta_x}=\sup_{x, y\in {\mathbb R}^d}\frac{[{\mathbf E} |{\mathbf E}^{\mathscr G}  X(x)-{\mathbf E}^{\mathscr G}  X(y)|_{{\mathcal H}}^p]^{1/p}}{|x-y|^\beta} \leqslant \|X\|_{C^\beta_x}. 
$$
For \eqref{Eq-DE}, we only give the proof for $k=1$. Again by Jensen's inequality, 
\begin{align*}
&\left|{\mathbf E}^{\mathscr G}  X(x+h)-{\mathbf E}^{\mathscr G} X(x)- [{\mathbf E}^{\mathscr G} \nabla X(x)]\cdot h\right|_{{\mathcal H}}\\
\leqslant & {\mathbf E}^{\mathscr G} \left|X(x+h)- X(x)- \nabla X(x)\cdot h\right|_{\mathcal H}. 
\end{align*}
Thus, 
\begin{align*}
&{\mathbf E} \, \left|{\mathbf E}^{\mathscr G}  X(x+h)-{\mathbf E}^{\mathscr G} X(x)- [{\mathbf E}^{\mathscr G} \nabla X(x)]\cdot h\right|_{{\mathcal H}}^p\\
\leqslant & {\mathbf E}\,  {\mathbf E}^{\mathscr G}  \left( \left|X(x+h)- X(x)- \nabla X(x)\cdot h\right|_{{\mathcal H}}^p \right )\\
=& \left|X(x+h)- X(x)- \nabla X(x)\cdot h\right |_{{\mathcal B}}^p\to 0\  (h\to \infty), 
\end{align*}
which gives the desired result. 
\end{proof}

\begin{lemma}\label{Le-Exchange}
Suppose $f: {\mathbb R}^n \times\Omega\to {\mathbb R}^m$ is ${\mathscr B}\times {\mathscr F}$ measurable and $f\in C^1({\mathbb R}^n; {\mathbb D}^{1, p})$,  then $\nabla f\in C({\mathbb R}^n; {\mathbb D}^{1, p})$ and 
\begin{equation}
\nabla D f =D  \nabla f 
\end{equation}
\end{lemma}
\begin{proof}
We assume $n=m=1$ for simple. For any $x\in {\mathbb R}$, by definition 
$$
\partial_{x, \theta} f(x):=\frac{f(x+\theta)-f(x)}{\theta}\overset{L^p(\Omega)}{\longrightarrow}  \partial_x f(x)\ (\theta\to0). 
$$
On the other hand, since $D f\in C^1({\mathbb R}; L^p(\Omega, H))$, we have 
\begin{align*}
D \partial_{x, \theta} f(x) =\frac{Df(x+\theta)-Df(x)}{\theta} \overset{L^p(\Omega; H)}{\longrightarrow}  \partial_x(D f)(x)\ (\theta\to0). 
\end{align*}
By the closability of Mallivian derivate, we get $ D\partial_xf(x)=\partial_xDf(x)\in {\mathbb D}^{1, p}$ and 
$$
\|\partial_xf(x)\|_{{\mathbb D}^{1, p}}= \liminf_{|\theta|\to0}  \left\|\frac{f(x+\theta)-f(x)}{\theta}\right\|_{{\mathbb D}^{1, p}} \leqslant \|f\|_{C^1({\mathbb R}; {\mathbb D}^{1, p})}. 
$$
\end{proof}

For any $F\in {\mathbb D}^{1,2}$, we have the following remarkable Clark-Ocone formula, 
\begin{equation}\label{Eq-CO-fml}
F={\mathbf E} (F)+\int_{0}^{1} {\mathbf E}^tD_t F\cdot d  W_t
:={\mathbf E} (F)+\sum_{k=1}^{d} \int_{0}^{1} {\mathbf E}\left(D_{t}^{k} F | \mathscr{F}_{t}\right) d  W_t^{k}. 
\end{equation}
\eqref{Eq-CO-fml} implies the following simple lemma. 
\begin{lemma}\label{Le-COF}
Suppose $F\in {\mathbb D}^{1,2}$, then for each $t\in[0,1]$, 
\begin{equation}\label{eq-EtF}
{\mathbf E}^t F= {\mathbf E} F + \int_0^t {\mathbf E}^sD_s F\cdot \mathrm{d}  W_s. 
\end{equation}
\end{lemma}
\begin{proof}
By Clark-Ocone's formula, 
$$
m_t= {\mathbf E} F + \int_0^t {\mathbf E}^sD_s F \cdot \mathrm{d}  W_s, 
$$
is a ${\mathscr F}_t$-martingale with $m_1=F$.  Thus,  
$$
{\mathbf E}^t F= {\mathbf E}^t m_1=m_t={\mathbf E} F + \int_0^t {\mathbf E}^sD_s F \cdot \mathrm{d}  W_s. 
$$
\end{proof}

The following Lemma is a modification of Theorem 1.1 in \cite{kunita1981some}, which is need in our proof of main result. Similar result on distributional valued processes can be found in  \cite{krylov2011ito}. 
\begin{lemma}[It\^o-Wentzell's formula]\label{Le-Ito}
Let  $(\Omega, {\mathscr F}, {\mathscr F}_t, {\mathbf P})$ be a standard filtered probability space satisfying the common conditions, $p, p'\in [1,\infty]$ with $\frac{1}{p}+\frac{1}{p'}=1$ and $\alpha_1, \alpha_2\in (0,1)$ with $\alpha_1+\alpha_2>1$. Suppose $X_t=(X_t^1,\cdots, X_t^n)$ be continuous semimartingales and  $\phi_t(x)$ be a random field continuous in $(x,t)\in Q$ almost surely. Assume $\phi$ and $X$ satisfy 
\begin{enumerate}
\item for each $t\in [0,1]$, ${\mathbb R}^n\ni x\mapsto \phi_t(x)\in {\mathbb R}$ is $C^2$ continuous a.s., 
\item for each $x\in {\mathbb R}^n$, $t\mapsto \phi_t(x)$ is a continuous ${\mathscr F}_t$-semimartingale represented as 
\begin{equation}\label{eq-phi}
 \phi_t(x)=\phi_0(x)+\int_0^t g_s(x)  \mathrm{d} s+ \int_0^t v^k_s(x) \mathrm{d} m^k_s, 
\end{equation}
\end{enumerate}
where $m^1, \cdots, m^d$ are continuous martingales, and the random field $g, v$ are locally bounded and 
\begin{enumerate}[(a)]
\item for each $x\in {\mathbb R}^n$, $t\mapsto g_t(x)$ and $t\mapsto v_t(x)$ are ${\mathscr F}_t$-adapted processes; 
\item for each $t\in[0,1]$, $x\mapsto v_t(x)$ is $C^1$ a.s.; 
\item for each $t\in [0,1]$, $x\mapsto g_t(x)$ is continuous, 
$$
 {\mathbf E} \sup_{x\in B_N} \left|\nabla \int_{t_1}^{t_2} g_s(x) \mathrm{d} s\right|^p \lesssim_{p, N} |t_1-t_2|^{\alpha_1p},
$$
$$
{\mathbf E}\left|X_{t_1\wedge\tau_N}-X_{t_2\wedge\tau_N}\right|^{p'}\lesssim_{p' N} |t_1-t_2|^{\alpha_2p'}, 
$$
where $\tau_N=\inf\{t>0: |X_t|>N\}$. 
\end{enumerate}
Then we have 
\begin{equation}\label{Eq-IW}
\begin{aligned} 
d \phi_{t}(X_t)=&g_t(X_t) \mathrm{d} t+ v_t^k(X_t) \mathrm{d} m_t^k+ \partial_i \phi_t(X_t) \mathrm{d} X^i_t\\
&+ \frac{1}{2} \partial_{ij} \phi_t(X_t) \mathrm{d} \langle X^i, X^j\rangle _t+ \partial_i v_t^k(X_t)\mathrm{d} \langle m^k, X^i\rangle _t
\end{aligned}
\end{equation}
\end{lemma} 
\begin{proof}
The proof is similar with Theorem 1.1 of \cite{kunita1981some}. Without loss of generality, we can assume $|X_t|$ is bounded by a constant $N$. For any $t>0$, let $t_l=l t/n, l=0, \cdots, n$. $s(n):=t[sn/t]/n$ and $X^n_{s}:=X_{s(n)}$. Then, 
\begin{align*}
\phi_t(X_t)-\phi_0(X_0)=& \sum_{l=0}^{n-1} [\phi_{t_{l+1}}(X_{t_l})-\phi_{t_l}(X_{t_l})]+ \sum_{l=0}^{n-1} [\phi_{t_{l+1}}(X_{t_{l+1}})-\phi_{t_{l+1}}(X_{t_l})]\\
=&: I_1^n+I_2^n. 
\end{align*}
By \eqref{eq-phi} and the definition of $X^n$
\begin{align*}
I_1^n = \sum_{l=0}^{n-1} \int_{t_l}^{t_{l+1}} g_s(X_{t_l}) \mathrm{d} s+\sum_{l=0}^{n-1} \int_{t_l}^{t_{l+1}} v^k_s(X_{t_l}) \mathrm{d} m^k_s=\int_0^t g_s(X^n_s) \mathrm{d} s+ \int_0^t v_s^k(X^n_s) \mathrm{d} m_s^k. 
\end{align*}
Since $g_s(X_s^n)\to g_s(X_s)$, $v_s(X_s^n)\to v_s(X_s)$ a.s. and $g, v$ are uniformly bounded in $[0,1]\times B_N$, we obtain 
$$
I_1^n \overset{{\mathbf P}}{\longrightarrow} \int_0^t g_s(X_s) \mathrm{d} s+ \int_0^t v_s^k(X_s) \mathrm{d} m_s^k,\ (n\to \infty). 
$$
By Taylor expansion, 
\begin{align*}
I^n_2=& \sum_{l=0}^{n-1} \partial_i\phi_{t_{l+1}}(X_{t_l})(X^i_{t_{l+1}}-X^i_{t_l}) + \frac{1}{2} \sum_{l=0}^{n-1} \partial_{ij} \phi_{t_{l+1}} (\xi_l) (X^i_{t_{l+1}}-X^i_{t_l}) (X^j_{t_{l+1}}-X^j_{t_l})\\
=&: I_{21}^n+ I_{22}^n, 
\end{align*}
where $\xi_l$ are some random variables between $X_{t_l}$ and $X_{t_{l+1}}$. 
It is standard to show that 
\begin{align*}
I_{22}^n \overset{{\mathbf P}}{\longrightarrow} \frac{1}{2} \int_0^t \partial_{ij}\phi_s(X_s) \mathrm{d} \langle X^i, X^j\rangle _s, \quad (n\to \infty). 
\end{align*}
For $I_{21}^n$, we rewrite it as 
\begin{align*}
I_{21}^n=& \sum_{i=0}^{n-1}\partial_i \phi_{t_l}(X_{t_l})(X^i_{t_{l+1}}-X^i_{t_l})+ \sum_{l=0}^{n-1}  [\partial_i \phi_{t_{l+1}}(X_{t_l})-\partial_i \phi_{t_l}(X_{t_l})] (X^i_{t_{l+1}}-X^i_{t_l})\\
=& : I_{211}^n+ I_{212}^n. 
\end{align*}
Like before, 
\begin{align*}
I^n_{211} \overset{{\mathbf P}}{\longrightarrow} \int_0^t \partial_i \phi_s(X_s) \mathrm{d} X^i_s, \quad (n\to \infty). 
\end{align*}
Again by \eqref{eq-phi}, 
\begin{align*}
I_{212}^n=& \sum_{l=0}^{n-1} \partial_i \left( \int_{t_l}^{t_{l+1}} g_s(X_{t_l}) \mathrm{d} s\right ) (X^i_{t_{l+1}}-X^i_{t_l}) + \sum_{l=0}^{n-1} \left(\int_{t_l}^{t_{l+1}} \partial_i v^k_s(X_{t_l}) \mathrm{d} m^k_s\right ) (X^i_{t_{l+1}}-X^i_{t_l})\\
=&: I^n_{2121}+I^n_{2122}. 
\end{align*}
By our assumption (c) and H\"older inequality, 
\begin{align*}
 {\mathbf E} |I^n_{2121} | \lesssim &\sum_{l=0}^{n-1}   \left[ {\mathbf E} \sup_{x\in B_N} \left|\nabla \int_{t_{l+1}}^{t_l} g_s(x) \mathrm{d} s\right|^p\right ]^{1/p} \left[{\mathbf E} \left|X_{t_{l+1}}-X_{t_l}\right|^{p'}\right ]^{1/p'}  \\
 \lesssim &\sum_{l=0}^{n-1} |t_{l+1}-t_l|^{\alpha_1+\alpha_2} \lesssim n^{-\alpha_1-\alpha_2+1}\to 0, \quad (k\to \infty). 
\end{align*}
It is standard to show 
$$
I_{2122}^n \overset{{\mathbf P}}{\longrightarrow} \int_0^t \partial_i v^k(X_s) \mathrm{d} \langle m^k, X^i\rangle _s, \quad (k\to \infty).
$$
Combine all the above calculations, we obtain \eqref{Eq-IW}. 
\end{proof}


\begin{ack}
The author would like to thank Professor Luo Dejun for raising this problem to him and also for having many useful discussions. 
\end{ack}

\begin{funding}
Research of G. Zhao is supported by the German Research Foundation (DFG) through the Collaborative Research Centre(CRC) 1283 ``Taming uncertainty and profiting from randomness and low regularity in analysis, stochastics and their applications".
\end{funding}



\begin{thebibliography}{99}

\bibitem{bahouri2011fourier}
Bahouri, H., Chemin, J.-Y. and Danchin, R.: 
\emph{Fourier analysis and nonlinear partial differential equations}, vol.~343.
\newblock Springer Science \& Business Media, 2011.

\bibitem{beck2019stochastic}
Beck, L., Flandoli, F., Gubinelli, M., and Maurelli, M.: Stochastic {ODE}s and stochastic linear {PDE}s with critical drift: regularity, duality and uniqueness.
\emph{Electronic Journal of Probability} {\bf 24} (2019), 1--72.

\bibitem{du2012lp}
Du, K., Qiu, J. and Tang, S.: ${L}^p$ theory for super-parabolic backward stochastic partial differential equations in the whole space.
\emph{Applied Mathematics \& Optimization } {\bf 65} (2012), no.~2, 175--219.

\bibitem{duboscq2016stochastic}
Duboscq, R. and R{\'e}veillac, A.: Stochastic regularization effects of semi-martingales on random functions.
\emph{Journal de Math{\'e}matiques Pures et Appliqu{\'e}es} {\bf 106} (2016), no.~6, 1141--1173.

\bibitem{fedrizzi2011pathwise}
Fedrizzi, E., and Flandoli, F.: Pathwise uniqueness and continuous dependence for SDEs with  non-regular drift. 
\emph{Stochastics: An International Journal of Probability and
  Stochastic Processes} {\bf 83} (2011), no.~3,  241--257.

\bibitem{flandoli2011random}
{\sc Flandoli, F.}
\newblock {\em Random Perturbation of {PDE}s and Fluid Dynamic Models:
  {\'E}cole d'{\'e}t{\'e} de Probabilit{\'e}s de Saint-Flour XL--2010},
  vol.~2015.
\newblock Springer Science \& Business Media, 2011.

\bibitem{flandoli2010well}
Flandoli, F., Gubinelli, M. and Priola, E.: Well-posedness of the transport equation by stochastic perturbation. 
\emph{Inventiones mathematicae} {\bf 180}  (2010), no.~1, 1--53.

\bibitem{john1978partial}
John, F.: 
\newblock {\em Partial Differential Equations}, vol.~1.
\newblock Springer, 1978.

\bibitem{krylov1999kolmogorov}
{K}rylov, N.:  {O}n {{K}}olmogorov's equations for finite dimensional diffusions. In \emph{Stochastic PDE’s and Kolmogorov Equations in Infinite Dimensions}, pp. 1-63,  Springer, Berlin, Heidelberg, 1999.


\bibitem{krylov2011ito}
Krylov, N.~V.: On the {I}t{\^o}-{W}entzell formula for distribution-valued processes
  and related topics.
\emph{Probability Theory and Related Fields} {\bf 150} (2011), no.~1-2, 295--319.

\bibitem{krylov2005strong}
Krylov, N.~V. and {R}\"ockner, M.: 
Strong solutions of stochastic equations with singular time dependent drift. 
\emph{Probability Theory and Related Fields} {\bf 131} (2005), no.~2, 154--196.

\bibitem{krylov2020strong}
Krylov, N.~V.: On strong solutions of {I}t\^o's equations with ${A}\in {W}^1_d$ and
  $b\in {L}_d$, preprint 2020, \href{https://arxiv.org/abs/2007.06040}{arXiv:2007.06040v1}.

\bibitem{kunita1981some}
Kunita, H.: Some extensions of ito's formula.  In \emph{S{\'e}minaire de Probabilit{\'e}s XV 1979/80}, pp.~118--141, Springer, 1981, 

\bibitem{mastrolia2017malliavin}
Mastrolia, T., Possama{\"\i}, D., and R{\'e}veillac, A.: On the {M}alliavin differentiability of {BSDE}s.
\emph{Annales de l'Institut Henri Poincar{\'e}, Probabilit{\'e}s et
  Statistiques} {\bf 53} (2017), no.~1, 464--492.

\bibitem{menoukeu2013variational}
Menoukeu-Pamen, O., Meyer-Brandis, T., Nilssen, T., Proske, F., and Zhang,
  T.: A variational approach to the construction and {M}alliavin
  differentiability of strong solutions of {{SDE}}'s.
\emph{Mathematische Annalen} {\bf 357} (2013), no.~2, 761--799.

\bibitem{rockner2021critical}
Röckner, M. and Zhao, G.: SDEs with critical time dependent drifts: strong solutions, preprint 2021, \href{https://arxiv.org/abs/2103.05803}{arXiv:2103.05803}. 

\bibitem{stroock1969diffusion}
Stroock, D.~W., and Varadhan, S.~R.: Diffusion processes with continuous coefficients, I.
\emph{Communications on Pure and Applied Mathematics} {\bf 22} (1969), no.~3, 345--400.

\bibitem{tang2016cauchy}
Tang, S., and Wei, W.: On the cauchy problem for backward stochastic partial differential
  equations in h{\"o}lder spaces.
\emph{The Annals of Probability} {\bf 44}  (2016), no.~1, 360--398.

\bibitem{veretennikov1980strong2} Veretennikov, A.~Y.: On strong solutions and explicit formulas for solutions of stochastic
  integral equations. 
\emph{Matematicheskii Sbornik} {\bf 153} (1980), no.~3, 434--452.

\bibitem{xia2020lqlp}
Xia, P., Xie, L., Zhang, X., and Zhao, G.: ${L}^q({L}^p)$-theory of stochastic differential equations. 
\emph{Stochastic Processes and their Applications} {\bf 130} (2020), no.~8, 5188--5211.

\bibitem{zhang2011stochastic}
Zhang, X.: {S}tochastic homeomorphism flows of SDEs with singular drifts and Sobolev diffusion coefficients. 
 \emph{Electronic {J}ournal of {P}robability} {\bf 16} (2011), 1096--1116.

\bibitem{zhang2016stochastic}
{Z}hang, X.:  {S}tochastic differential equations with {S}obolev diffusion and
  singular drift and applications.
\emph{The {A}nnals of {A}pplied {P}robability} {\bf 26} (2016), no.~5, 2697--2732.

\bibitem{zhang2017heat}
Zhang, X., and Zhao, G.: Heat kernel and ergodicity of {SDE}s with distributional drifts, preprint 2017,   \href{https://arxiv.org/abs/1710.10537}{arXiv:1710.10537}.

\bibitem{zhang2020stochastic}
Zhang, X., and Zhao, G.: Stochastic Lagrangian path for leray's solutions of 3D Navier-Stokes equations. 
\emph{Communications in Mathematical Physics} {\bf 381}(2021), no.~2, 491--525.

\bibitem{zhao2019stochastic}
Zhao, G.: Stochastic {L}agrangian flows for {{SDE}}s with rough coefficients, preprint 2019,  \href{https://arxiv.org/abs/1911.05562}{arXiv:1911.05562}. 

\bibitem{zvonkin1974transformation}
Zvonkin, A.~K.: A transformation of the phase space of a diffusion process that removes the drift. 
\emph{Mathematics of the USSR-Sbornik} {\bf 22} (1974), no.~1, 129.







\end{thebibliography}
\end{document}